\documentclass[a4paper,12pt]{amsart}
\usepackage{amsmath}
\usepackage{amsmath}
\usepackage{amssymb}
\usepackage{amscd}
\usepackage{mathrsfs}
\usepackage[all]{xy}
\usepackage[dvipdfm]{graphicx}
\def\mathcal{\mathscr}
\everymath{\displaystyle}
\newtheorem{thm}{Theorem}[section]
\newtheorem{lem}[thm]{Lemma}
\newtheorem{cor}[thm]{Corollary}
\newtheorem{prop}[thm]{Proposition}

\theoremstyle{definition}

\newtheorem{rem}[thm]{Remark}

\newtheorem{defn}[thm]{Definition}

\newcommand{\mca}[1]{{\mathcal{#1}}}

\def\Z{{\mathbb Z}}

\def\R{{\mathbb R}}
\def\con{\text{\rm con}}
\def\id{\text{\rm id}}
\def\Image{\text{\rm Im}}
\def\ind{\text{\rm ind}}

\def\len{\text{\rm length}}

\def\conv{\text{\rm conv}}
\def\diam{\text{\rm diam}}
\def\dim{\text{\rm dim}}
\def\dist{\text{\rm dist}}
\def\ep{\varepsilon} 
\def\ev{\text{\rm ev}} 
\def\interior{\text{\rm int}}
\def\plim{\varprojlim}
\def\supp{\text{\rm supp}}
\def\Spec{\text{\rm Spec}}
\def\Star{\text{\rm Star}}
\def\wid{\text{\rm width}}

\begin{document}
\pagestyle{plain}
\thispagestyle{plain}

\title[Periodic billiard trajectories and Morse theory on loop spaces]
{Periodic billiard trajectories and Morse theory on loop spaces}

\author[Kei Irie]{Kei Irie}
\address{Research Institute for Mathematical Sciences, Kyoto University,
Kyoto 606-8502, Japan}
\email{iriek@kurims.kyoto-u.ac.jp}

\date{\today}

\begin{abstract}
We study periodic billiard trajectories on a compact Riemannian manifold with boundary, by applying Morse theory to 
Lagrangian action functionals on the loop space of the manifold. 
Based on the approximation method due to Benci-Giannoni, we prove that nonvanishing of relative homology of a certain pair of loop spaces 
implies the existence of a periodic billiard trajectory. 
We also prove a parallel result for path spaces. 
We apply those results to show the existence of short billiard trajectories and short geodesic loops. 
We also recover two known results on the length of a shortest periodic billiard trajectory on a convex body: Ghomi's inequality, and 
Brunn-Minkowski type inequality due to Artstein-Ostrover. 
\end{abstract}

\maketitle

\section{Introduction and results} 

In this section, we describe our main results and the plan of this paper. 

\subsection{Definitions of periodic/brake billiard trajectory}
First let us fix the definition of periodic billiard trajectory. 
We also introduce the notion of brake billiard trajectory, which is a relative version of periodic trajectory. 

Let $Q$ be a Riemannian manifold with boundary. 
We set $S^1:=\R/\Z$. 
A nonconstant continuous map $\gamma:S^1 \to Q$ is called a \textit{periodic billiard trajectory}, if there exists a finite set $\mca{B}_\gamma \subset S^1$ such that 
$\ddot{\gamma} \equiv 0$ on $S^1 \setminus \mca{B}_\gamma$, and every $t \in \mca{B}_\gamma$ satisfies the following conditions:
\begin{itemize}
\item[B-(i):] $\gamma(t) \in \partial Q$. 
\item[B-(ii):] $\dot{\gamma}^\pm (t):= \lim_{h \to 0 \pm} \dot{\gamma}(t+h)$ satisfies the following equation:
\[ 
\dot{\gamma}^+(t)  + \dot{\gamma}^-(t)  \in T_{\gamma(t)} \partial Q, \qquad
\dot{\gamma}^+(t)  - \dot{\gamma}^-(t)  \in (T_{\gamma(t)} \partial Q)^{\perp} \setminus \{0\}.
\] 
This equation is called the \textit{law of reflection}. 
\end{itemize} 

\begin{rem}
Here are some remarks on the above definition. 
\begin{itemize}
\item 
A periodic billiard trajectory $\gamma$ might be a closed geodesic on $Q$. 
 In that case, $\mca{B}_\gamma  =\emptyset$.
\item 
If $\gamma$ touches $\partial Q$ at $t \in S^1$, B-(ii) does not hold since 
$\dot{\gamma}^+(t) - \dot{\gamma}^-(t)=0$. 
Therefore, $\gamma^{-1}(\partial Q)$ might be strictly larger than $\mca{B}_\gamma$. 
\item 
The law of reflection implies that, $|\dot{\gamma}|$ is constant on $S^1 \setminus \mca{B}_\gamma$. 
Moreover, $|\dot{\gamma}| \ne 0$ since $\gamma$ is a nonconstant map.
\end{itemize} 
\end{rem}

A nonconstant continuous map $\gamma :[0,1] \to Q$ is called a \textit{brake billiard trajectory},  if it satisfies the following conditions:
\begin{itemize}
\item There exists a finite set $\mca{B}_\gamma \subset (0,1)$ such that 
$\ddot{\gamma} \equiv 0$ on $[0,1] \setminus \mca{B}_\gamma$, and every $t \in \mca{B}_\gamma$ satisfies B-(i), B-(ii). 
\item  $\gamma(0), \gamma(1) \in \partial Q$, and $\dot{\gamma}(0)$, $\dot{\gamma}(1)$ are perpendicular to $\partial Q$. 
\end{itemize} 
The name "brake" billiard trajectory comes from the notion of brake orbit in classical mechanics
(see \cite{HZ} pp.131). 
In both (periodic and brake) cases, elements of $\mca{B}_\gamma$ are called \textit{bounce times} of $\gamma$. 

For any brake billiard trajectory $\gamma:[0,1] \to Q$, we have a periodic billiard trajectory $\Gamma: S^1 \to Q$ which is defined by 
\[
\Gamma(t): = \begin{cases} \gamma(2t) &(0 \le t \le 1/2) \\ \gamma(2-2t) &(1/2 \le t \le 1). \end{cases}
\]
This is a genuine billiard trajectory, i.e. $\mca{B}_\Gamma \ne \emptyset$. 
If $\gamma$ satisfies $\mca{B}_\gamma = \emptyset$, 
$\Gamma$ is called a \textit{bouncing ball orbit}. 

\subsection{Billiard trajectory and topology of path/loop spaces} 
We state our first result Theorem \ref{mainthm}, 
which claims that 
nonvanishing of relative homology of a certain pair of loop spaces 
implies 
the existence of a periodic billiard trajectory. 
We also prove a parallel result for brake billiard trajectories. 

First we fix notations. 
A continuous map $\gamma:S^1 \to Q$ is of class $W^{1,2}$, 
if it is absolutely continuous and its first derivative is square-integrable. 
$W^{1,2}(S^1,Q)$ denotes the space of $W^{1,2}$-maps $S^1 \to Q$. 
$W^{1,2}([0,1],Q)$ is defined in the same way. 
We use the following abbreviations:
\[
\Lambda(Q):= W^{1,2}([0,1], Q), \qquad
\Omega(Q):= W^{1,2}(S^1, Q). 
\]
These spaces are equipped with the natural topologies.
For any subset $S \subset Q$, we set 
\[
\Lambda(S):= \{ \gamma \in \Lambda(Q) \mid \gamma([0,1]) \subset S\}, \qquad
\Omega(S):= \{ \gamma \in \Omega(Q) \mid \gamma(S^1) \subset S \}. 
\]
They are equipped with the induced topologies as subsets of  $\Lambda(Q)$, $\Omega(Q)$. 

We define $\mca{E}: \Lambda(Q) \to \R$ by $\mca{E}(\gamma):= \int_0^1  \frac{|\dot{\gamma}(t)|^2}{2} \, dt$. 
$\mca{E}: \Omega(Q) \to \R$ is defined in the same way. 
For any $a \in \R$, we define
\[
\Lambda^a(Q):= \{ \gamma \in \Lambda(Q) \mid \mca{E}(\gamma)<a \}, \qquad 
\Omega^a(Q):= \{ \gamma \in \Omega(Q) \mid \mca{E}(\gamma) < a\}. 
\]
When $a<b$, one has obvious inclusions 
$\Lambda^a(Q) \subset \Lambda^b(Q)$, 
$\Omega^a(Q) \subset \Omega^b(Q)$. 
Let $\delta$ be any positive number. We denote the distance on $Q$ by $\dist$, and define 
\begin{align*}
Q(\delta)&:=\{ q \in Q \mid \dist(q,\partial Q) \ge \delta \}, \\
\Lambda_\delta(Q)&:= \Lambda(Q) \setminus \Lambda(Q(\delta)) = \{ \gamma \in \Lambda(Q) \mid \dist(\gamma([0,1]), \partial Q) < \delta \}, \\
\Omega_\delta(Q)&:= \Omega(Q) \setminus \Omega(Q(\delta)) = \{ \gamma \in \Omega(Q) \mid \dist(\gamma(S^1),\partial Q) < \delta \}.
\end{align*}
When $\delta'  < \delta$, one has obvious inclusions 
$\Lambda_{\delta'}(Q) \subset \Lambda_\delta(Q)$, 
$\Omega_{\delta'}(Q) \subset \Omega_\delta(Q)$. 

\begin{thm}\label{mainthm}
Let $Q$ be a compact Riemannian manifold with boundary, $a<b$ be positive real numbers, and $j$ be a nonnegative integer. 
\begin{enumerate}
\item[(i):]
If $\varprojlim_{\delta \to 0} H_j(\Lambda^b(Q) \cup \Lambda_\delta(Q), \Lambda^a(Q) \cup \Lambda_\delta(Q)) \ne 0$, 
there exists a brake billiard trajectory $\gamma$ on $Q$, such that $\sharp \mca{B}_\gamma \le j-2$, $\len(\gamma) \in [\sqrt{2a}, \sqrt{2b}]$. 
\item[(ii):]
If $\varprojlim_{\delta \to 0} H_j(\Omega^b(Q) \cup \Omega_\delta(Q), \Omega^a(Q) \cup \Omega_\delta(Q)) \ne 0$, 
there exists a periodic billiard trajectory $\gamma$ on $Q$,  such that $\sharp \mca{B}_\gamma \le j$, $\len(\gamma) \in [\sqrt{2a}, \sqrt{2b}]$. 
\end{enumerate}
\end{thm}

\begin{rem}\label{rem:closed}
Let us check Theorem \ref{mainthm} when $Q$ is a closed manifold. 
In this case, there always holds $H_*(\Lambda^b(Q), \Lambda^a(Q))=0$, thus the assumption of (i) is never satisfied. 
On the other hand, (ii) claims that, if $H_*(\Omega^b(Q), \Omega^a(Q)) \ne 0$ then there exists a closed geodesic $\gamma$ on $Q$ 
such that $\len(\gamma) \in [\sqrt{2a}, \sqrt{2b}]$: this is a very well-known fact in the study of closed geodesics. 
Thus the main point of Theorem \ref{mainthm} is when $Q$ has nonempty boundary, and 
one can see it as the billiard version of the above-mentioned classical fact. 
\end{rem} 

We explain the idea of the proof. 
For simplicity, we only discuss the case (i). 
We take a "potential function" $U: \interior Q \to \R_{\ge 0}$ which diverges to $\infty$ near $\partial Q$. 
We also take $\ep>0$, and study the following equation for $\gamma: [0,1] \to \interior Q$:
\begin{equation}\label{eq:sketch}
\dot{\gamma}(0) = \dot{\gamma}(1)=0, \qquad \ddot{\gamma}(t) + \ep \nabla U(\gamma(t) )  \equiv 0. 
\end{equation}
As is well-known, solutions of this equation are critical points of the Lagrangian functional $\mca{L}_\ep$ on the path space $\Lambda(\interior Q)$, 
which is defined as
\[
\mca{L}_\ep(\gamma):= \int_0^1  \frac{|\dot{\gamma}(t)|^2}{2} - \ep U(\gamma(t)) \, dt.
\]
One can prove the existence of a solution of (\ref{eq:sketch}) by Morse theory for the functional $\mca{L}_\ep$. 
The precise statement is Proposition \ref{prop:Morsetheory}, and Section 2 is devoted to its proof. 

Suppose that we have a solution $\gamma_\ep$ of (\ref{eq:sketch}) for sufficiently small any $\ep>0$, which satisfies certain estimates on 
$\mca{L}_\ep(\gamma_\ep)$ and the Morse index. 
Then, we can get a billiard trajectory $\gamma$ as a limit of $\gamma_\ep$ as $\ep \to 0$, which satisfies corresponding estimates on $\len(\gamma)$ and 
$\sharp \mca{B}_\gamma$. 
The precise statement is Proposition \ref{prop:limit}, and Section 3 is devoted to its proof. 
Combining results in Sections 2 and 3, we will complete the proof of Theorem \ref{mainthm} in Section 4. 

The above strategy of the proof is heavily influenced by \cite{BG}. 
In particular, our arguments in Sections 2 and 3 closely follow the arguments in \cite{BG}. 
Nevertheless, we explain most details for the reader's convenience.

\subsection{Short billiard trajectory} 
As an application of Theorem \ref{mainthm}, we prove the existence of short billiard trajectories. 
First let us state the result. 
Let $Q$ be a compact, connected Riemannian manifold with nonempty boundary. 
$r(Q)$ denotes the \textit{inradius} of $Q$, i.e. $r(Q):= \max_{q \in Q} \dist(q,\partial Q)$. 
It is easy to see that $r(Q) < \infty$. 

\begin{thm}\label{thm:short} 
Let $j$ be a positive integer such that $H_j(Q,\partial Q:\Z) \ne 0$. 
Then, there exist following billiard trajectories on $Q$: 
\begin{itemize}
\item A brake billiard trajectory $\gamma_B$, such that $\sharp \mca{B}_{\gamma_B}  \le j-1$, $\len(\gamma_B) \le 2j r(Q)$. 
\item A periodic billiard trajectory $\gamma_P$, such that $\sharp \mca{B}_{\gamma_P} \le j+1$, $\len(\gamma_P) \le 2(j+1)r(Q)$. 
\end{itemize}
\end{thm} 
\begin{rem}
The author knows very little about examples in which the above estimates are sharp. 
It is easy to see that the estimates are sharp for $j=1$: consider the case $Q$ is a line segment. 
For $j=2$, estimates $\sharp \mca{B}_{\gamma_B} \le 1$ and $\sharp \mca{B}_{\gamma_P} \le 3$ are sharp:
there exists a planar domain which does not contain any bouncing ball orbits, see Figure 1-(b) in \cite{Gh}.
\end{rem}

Theorem \ref{thm:short} is proved in Section 5. 
In this subsection, we explain some consequences of Theorem \ref{thm:short}.  
Let us introduce the following notations:  
\begin{align*}
\mu_B(Q)&:= \inf \{ \len(\gamma) \mid \gamma: \text{brake billiard trajectory on $Q$} \}, \\
\mu_P(Q)&:= \inf \{ \len(\gamma) \mid \gamma: \text{periodic billiard trajectory on $Q$} \}.
\end{align*}
As an immediate consequence of Theorem \ref{thm:short}, we obtain the following estimate. 

\begin{cor}\label{cor:short}
Let $n$ denote the dimension of $Q$. 
Then, there holds $\mu_B(Q) \le 2n r(Q)$, $\mu_P(Q) \le 2(n+1) r(Q)$. 
\end{cor} 

The above estimate of  $\mu_P$ was already proved in \cite{ArOs} for convex domains in $\R^n$, and 
in \cite{Ir}  for arbitrary domains in $\R^n$. 
For other previous results on short periodic billiard trajectories, see \cite{ArOs} Section 1,  and references therein. 

Another consequence of Theorem \ref{thm:short} is a new proof of the following result on short geodesic loops, which is due to \cite{Rot}. 
The original proof in \cite{Rot} is based on the Birkhoff curve shortening process, and seems quite different from our arguments. 

\begin{cor}[Rotman \cite{Rot}]\label{cor:gdloop}
Let $M$ be a closed Riemannian manifold, $p \in M$, and $j$ be a positive integer. 
If $\pi_j(M,p) \ne 0$, there exists a nonconstant geodesic loop $\gamma$ at $p$
(i.e. a geodesic path $\gamma:[0,1] \to M$ such that $\gamma(0), \gamma(1)=p$) 
such that $\len(\gamma) \le 2j \diam (M)$. 
\end{cor} 

Our idea to prove Corollary \ref{cor:gdloop}  is to take a short brake billiard trajectory on $\{ x \in M \mid \dist(x,p) \ge \ep\}$ and let $\ep \to 0$. 
Details will be explained in Section 5.3. 

\subsection{The length of a shortest periodic billiard trajectory on a convex body}
A convex set $K \subset \R^n$ is called a \textit{convex body}, if $K$ is compact and $\interior K \ne \emptyset$. 
It is possible to show that, for any convex body with smooth boundary $K$, there exists a periodic billiard trajectory on $K$ of length $\mu_P(K)$
(see Remark \ref{rem:shortest}).

Let us recall two remarkable geometric inequalities on $\mu_P$ of convex bodies, 
which are proved in \cite{ArOs} and \cite{Gh}.
In Section 6, we recover these results using our method. 
A recent paper \cite{ABKS} obtains similar proofs based on the technique in \cite{BB}, 
 in a more general setting of Finsler billiards. 

The first one is the Brunn-Minkowski type inequality \cite{ArOs}.
For any two convex bodies $K_1, K_2 \subset \R^n$, their \textit{Minkowski sum} 
$K_1+K_2:= \{x_1 + x_2 \mid x_1 \in K_1, x_2 \in K_2\}$ is again a convex body. 
The following result is proved in \cite{ArOs}, based on their Brunn-Minkowski type inequality for 
symplectic capacity \cite{ArOs_old}.

\begin{thm}[Artstein-Ostrover \cite{ArOs}]\label{thm:BruMin}
Let $K_1, K_2$ be convex bodies in $\R^n$. 
Suppose that $K_1$, $K_2$ and $K_1+K_2$ have smooth boundaries. Then 
\[
\mu_P(K_1+K_2) \ge \mu_P(K_1) + \mu_P(K_2).
\]
Equality holds if and only if there exists a closed curve which, up to parallel displacement and scaling, 
is a shortest periodic billiard trajectory of both $K_1$ and $K_2$. 
\end{thm} 

The second one is a lower estimate of $\mu_P$ by inradius, which is proved in \cite{Gh} by beautiful elementary arguments.
$\wid (K)$ denotes the thickness of the narrowest slab which contains $K$. 

\begin{thm}[Ghomi \cite{Gh}]\label{thm:Ghomi}
For any convex body $K \subset \R^n$ with smooth boundary, there holds 
$\mu_P(K) \ge 4r(K)$. 
Equality holds if and only if $2r(K)=\wid (K)$. 
In this case, every shortest periodic billiard trajectory on $K$ is a bouncing ball orbit.
\end{thm} 

\begin{rem}
We only partially recover original results in \cite{ArOs} and \cite{Gh}. 
In \cite{ArOs}, the authors prove Theorem \ref{thm:BruMin} in a more general setting of Minkowski billiards, whereas 
we will discuss only Euclidean billiards. 
On the other hand, \cite{Gh} does not assume the smoothness of $\partial K$. 
\end{rem} 

\section{Approximating problem}
In this section, we study an approximating problem for the billiard problem, which was introduced in \cite{BG}. 
In Section 2.1, we fix the setting and state Proposition \ref{prop:Morsetheory}, which is a main result in this section. 
Section 2.2 is devoted to its proof. 

Throughout this and the next sections, $Q$ denotes a compact, connected Riemannian manifold with nonempty boundary. 
We abbreviate $\Lambda(\interior Q)$, $\Omega(\interior Q)$ as $\Lambda$, $\Omega$. 
These spaces have natural structures of smooth Hilbert manifolds. 
For any $\gamma_\Lambda \in \Lambda$ and $\gamma_\Omega \in \Omega$, tangent spaces at $\gamma_\Lambda$ and $\gamma_\Omega$ 
are described as 
\[
T_{\gamma_\Lambda} \Lambda = W^{1,2}([0,1], \gamma_\Lambda^* (TQ)), \qquad
T_{\gamma_\Omega} \Omega = W^{1,2}(S^1, \gamma_\Omega^*(TQ)).
\]

\subsection{Setting} 

We take and fix  $\rho \in C^\infty(\R_{\ge 0})$ such that: 
\begin{itemize}
\item $\rho(t)=t$ for any $0 \le t \le 1$. 
\item $0 \le \rho(t) \le 2$, $0 \le \rho'(t) \le 1$ for any $t \ge 0$.
\item $\rho(t) =2$ for any $t \ge 3$. 
\end{itemize}

We define $d \in C^0(Q)$ by $d(q):= \dist(q, \partial Q)$. 
Recall the notation $Q(\delta):= \{ q \in Q \mid d(q) \ge \delta \}$ in Section 1.2. 
When $\delta>0$ is sufficiently small, $d$ is of $C^\infty$ and satisfies $|\nabla d| \equiv 1$ on $Q \setminus Q(3\delta)$.
For such $\delta$, we define $h_\delta \in C^\infty(Q)$ and $U_\delta \in C^\infty(\interior Q)$ by 
\[
h_\delta(q):= \delta\rho(d(q)/\delta), \qquad U_\delta(q):= h_\delta(q)^{-2} - (2\delta)^{-2}. 
\]
In this and the next sections, we fix $\delta$ and abbreviate $h_\delta$, $U_\delta$ as $h$, $U$. 
The following lemma is easy to prove, and we will use it a few times. 

\begin{lem}\label{lem:U}
\begin{enumerate}
\item[(i):] Let $\nu$ be a smooth vector field on $Q$ such that $\nu \equiv - \nabla d$ on $Q \setminus Q(3\delta)$. 
Then $|\nabla U(q)| = \langle  \nabla U(q), \nu(q) \rangle$ for any $q \in \interior Q$. 
\item[(ii): ] There holds $\lim_{q \to \partial Q} U(q)/|\nabla U(q)| = 0$.
\end{enumerate}
\end{lem}

First we consider the approximating problem for brake billiard trajectory. 
Suppose that 
$V \in C^\infty([0,1] \times \interior Q)$ satisfies the following property: 
\begin{itemize}
\item[V-(i):] There exists $\ep>0$ and a compact set $K \subset \interior Q$ such that $V(t,q)= \ep U(q)$ for any $t \in [0,1]$, $q \notin K$. 
\end{itemize}
We define $L^\Lambda_V \in C^\infty([0,1] \times T(\interior Q) )$ and 
$\mca{L}^\Lambda_V: \Lambda \to \R$  by 
\[
L^\Lambda_V(t,q,v):= \frac{|v|^2}{2} - V(t,q) \, (q \in \interior Q, v \in T_q Q), \quad 
\mca{L}^\Lambda_V(\gamma):= \int_0^1 L^\Lambda_V(t,\gamma, \dot{\gamma}) \, dt. 
\]
$\mca{L}^\Lambda_V$ is a $C^\infty$ functional on $\Lambda$, and its differential is computed as 
\[
d \mca{L}^\Lambda_V(\gamma)(\zeta) = \int_0^1 \langle \dot{\gamma}, \nabla_t \zeta  \rangle - dV_t(\zeta) \, dt \qquad( \zeta \in T_\gamma \Lambda), 
\]
where $\nabla_t$ denotes the Levi-Civita covariant derivative, and $V_t \in C^\infty(\interior Q)$ is defined by $V_t(q):=V(t,q)$. 
$\gamma \in \Lambda$ satisfies $d\mca{L}^\Lambda_V(\gamma)=0$ if and only if it is of class $C^\infty$ and satisfies
\begin{equation}\label{eq:EL}
\dot{\gamma}(0)= \dot{\gamma}(1)=0, \qquad \ddot{\gamma}(t) + \nabla V_t(\gamma(t)) \equiv 0. 
\end{equation}

For any $\gamma$ satisfying (\ref{eq:EL}), the Hessian of $\mca{L}^\Lambda_V$ at $\gamma$ is given by the following formula, 
where $R$ denotes the curvature tensor. 
\begin{equation}\label{eq:Hessian} 
d^2 \mca{L}^\Lambda_V(\gamma)(\eta,\zeta) =\int_0^1 \langle \nabla_t \eta, \nabla_t \zeta  \rangle - \langle R(\dot{\gamma},\eta)(\zeta), \dot{\gamma} \rangle - \langle \nabla_\eta \nabla V_t(\gamma), \zeta \rangle 
\, dt \quad
(\eta, \zeta \in T_\gamma \Lambda). 
\end{equation}
$\ind(\gamma)$ denotes the Morse index of $\gamma$, that is the number of negative eigenvalues of $d^2\mca{L}^\Lambda_V(\gamma)$. 
As is well-known, $\ind(\gamma)< \infty$ (see e.g. \cite{AbS} Proposition 3.1 (iii)). 

Next we consider the approximating problem for periodic billiard trajectory. 
Suppose that $V \in C^\infty(S^1 \times \interior Q)$ satisfies the following property: 
\begin{itemize}
\item[V-(ii):] There exists $\ep>0$ and a compact set $K \subset \interior Q$ such that $V(t,q)=\ep U(q)$ for any $t \in S^1$, $q \notin K$. 
\end{itemize} 
We define $L^\Omega_V \in C^\infty (S^1 \times T(\interior Q))$ and $\mca{L}^\Omega_V: \Omega \to \R$ by
\[
L^\Omega_V(t,q,v)= \frac{|v|^2}{2} - V(t,q), \qquad 
\mca{L}^\Omega_V(\gamma):= \int_{S^1} L^\Omega_V(t,\gamma, \dot{\gamma})\, dt.
\]
$\gamma \in \Omega$ satisfies $d\mca{L}^\Omega_V(\gamma)=0$ if and only if it is of class $C^\infty$ and 
satisfies $\ddot{\gamma}(t) + \nabla V_t(\gamma(t)) \equiv 0$. 
The goal of this section is to prove the following proposition: 

\begin{prop}\label{prop:Morsetheory}
Let $a<b$ be real numbers, and $j$ be a nonnegative integer.
\begin{enumerate}
\item[(i):]
For any $V \in C^\infty([0,1] \times \interior Q)$ which satisfies V-(i) and 
\[
H_j(\{ \mca{L}^\Lambda_V<b \}, \{ \mca{L}^\Lambda_V< a\}) \ne 0,
\]
there exists $\gamma \in \Lambda$ such that 
$d\mca{L}^\Lambda_V(\gamma)=0$ and $\mca{L}^\Lambda_V(\gamma) \in [a,b]$, 
$\ind(\gamma) \le j$. 
\item[(ii):]
For any $V \in C^\infty(S^1 \times \interior Q)$ which satisfies  V-(ii) and 
\[
H_j(\{ \mca{L}^\Omega_V<b\}, \{ \mca{L}^\Omega_V<a\}) \ne 0,
\]
there exists $\gamma \in \Omega$ such that
$d\mca{L}^\Omega_V(\gamma)=0$ and $\mca{L}^\Omega_V(\gamma) \in [a,b]$, 
$\ind(\gamma) \le j$. 
\end{enumerate} 
\end{prop}

\subsection{Proof of Proposition \ref{prop:Morsetheory}}
We only prove (i), since (ii) can be proved by parallel arguments. 
In this subsection, we abbreviate $L^\Lambda_V$ and $\mca{L}^\Lambda_V$ as 
$L_V$, $\mca{L}_V$. 
As a first step, we need the following result: 

\begin{lem}\label{lem:boundary}
Let $(\gamma_j)_j$ be a sequence in $\Lambda$, such that 
$\lim_{j \to \infty} \dist(\gamma_j([0,1]), \partial Q) = 0$ and 
$\sup_j  \| \dot{\gamma}_j \|_{L^2} < \infty$. 
Then, there holds 
$\lim_{j \to \infty} \int_0^1 h(\gamma_j)^{-2} \, dt = \infty$. 
\end{lem}
\begin{proof}
See Lemma 3.6 in \cite{BG}. 
\end{proof} 

For any $\gamma \in \Lambda$ and $\eta, \zeta \in T_\gamma \Lambda$, we define a Riemannian metric $\langle \cdot, \cdot \rangle_\Lambda$ on $\Lambda$ by 
\begin{equation}\label{eq:metric}
\langle \eta, \zeta \rangle_\Lambda:= \int_0^1 \langle \eta, \zeta \rangle  + \langle \nabla_t \eta , \nabla_t \zeta \rangle  \, dt, \qquad
\| \eta \|_\Lambda:= \langle \eta, \eta \rangle_\Lambda^{1/2}.
\end{equation} 
$\| \cdot \|_\Lambda$ defines a distance function $d_\Lambda$ on $\Lambda$ in the obvious way. 
This metric structure on $\Lambda = \Lambda(\interior Q)$ 
naturally extends to $\Lambda(Q)$, 
and this makes $\Lambda(Q)$ a complete metric space. 
(Notice that $\Lambda (Q)$ is \textit{not} a Hilbert manifold, even with boundary.)

\begin{lem}\label{lem:complete}
For any interval $J \subset \R$, which is closed and bounded from below, 
$(\mca{L}_V^{-1}(J), d_\Lambda)$ is a complete metric space. 
\end{lem}
\begin{proof} 
Let $(\gamma_j)_j$ be a Cauchy sequence on $(\mca{L}_V^{-1}(J), d_\Lambda)$. 
There exists $\gamma_\infty \in \Lambda(Q)$ such that $\lim_{j \to \infty} d_\Lambda(\gamma_j, \gamma_\infty)=0$. 
It is enough to show that $\gamma_\infty([0,1]) \subset \interior Q$. 

Suppose that $\gamma_\infty([0,1])$ intersects $\partial Q$. 
Then, $\lim_{j \to \infty}  \dist(\gamma_j([0,1]), \partial Q)=0$. 
On the other hand, $\sup_j \| \dot{\gamma}_j \|_{L^2} < \infty$, since the convergence in $d_\Lambda$ implies a convergence in $W^{1,2}$-topology. 
Hence Lemma \ref{lem:boundary} implies $\lim_{j \to \infty} \mca{L}_V(\gamma_j) = -\infty$, contradicting the assumption that
$\mca{L}_V(\gamma_j) \in J$ for all $j$. 
\end{proof} 

Next we discuss the Palais-Smale (PS) condition for $\mca{L}_V$. 
For each $\gamma \in \Lambda$, we define $\nabla \mca{L}_V(\gamma) \in T_\gamma \Lambda$, the \textit{gradient vector} of $\mca{L}_V$ at $\gamma$ by 
\[
\langle \nabla \mca{L}_V(\gamma), \eta \rangle_\Lambda := d\mca{L}_V(\eta) \quad (\forall \eta \in T_\gamma \Lambda).
\]

\begin{defn}\label{defn:PS} 
Let $X$ be a (possibly infinite dimensional) Riemannian manifold and $f: X \to \R$ be a smooth function. 
A sequence $(x_j)_{j=1,2,\ldots}$ is a \textit{PS-sequence} of $f$, if $(f(x_j))_j$ is bounded and $\lim_{j \to \infty} \| \nabla f(x_j) \|=0$. 
We shall say that $f$ satisfies the \textit{PS-condition}, if any PS-sequence of $f$ contains a convergent subsequence. 
\end{defn} 

We are going to show that $\mca{L}_V$ satisfies the PS-condition. 
Our argument is based on the following result: 

\begin{lem}\label{lem:PS_closed}
Let $M$ be a closed Riemannian manifold, and suppose that $\Lambda(M)$ is equipped with a Riemannian metric 
in the same manner as (\ref{eq:metric}). 
For any $W \in C^\infty([0,1] \times M)$, 
\[
\mca{L}_W: \Lambda(M) \to \R; \qquad \gamma \mapsto \int_0^1 \frac{|\dot{\gamma}(t)|^2}{2} - W(t,\gamma(t)) \, dt
\]
satisfies the PS-condition. 
\end{lem}
\begin{proof}
The claim follows from Proposition 3.3 in \cite{AbS} (see also \cite{Ben}). 
\end{proof} 

Since our base manifold $\interior Q$ is open and $V$ diverges to $\infty$ near the boundary, we cannot apply Lemma \ref{lem:PS_closed} directly.
For this reason, we need the following lemma: 

\begin{lem}\label{lem:PS_boundary}
If $(\gamma_j)_j$ is a PS-sequence of $\mca{L}_V$, 
then $\inf_j \dist(\gamma_j([0,1]), \partial Q) >0$. 
\end{lem}
\begin{proof} 
Since $V$ satisfies V-(i), there exists $\ep>0$ such that $V(t,q) - \ep U(q)$ is compactly supported. 
Hence $\mca{L}_{\ep U}(\gamma_j)$, $\|\nabla \mca{L}_{\ep U}(\gamma_j)\|_\Lambda$ are both bounded on $j$. 
Let us take $\nu$ as in Lemma \ref{lem:U} (i), then there holds 
\[
\int_0^1 |\ep \nabla U(\gamma_j)| \, dt 
=\int_0^1 \langle \ep \nabla U(\gamma_j), \nu(\gamma_j) \rangle \, dt 
= - \langle \nabla \mca{L}_{\ep U}(\gamma_j), \nu(\gamma_j) \rangle_\Lambda  + \int_0^1 \langle \dot{\gamma}_j, \nabla_t(\nu(\gamma_j)) \rangle \, dt.
\]
We can bound RHS using the following obvious inequalities:
\[
\| \nu(\gamma_j) \|_{L^2} \le \max_{q \in Q} |\nu(q)|, \qquad
\| \nabla_t (\nu(\gamma_j)) \|_{L^2} \le \max_{q \in Q} |\nabla \nu (q)| \cdot \| \dot{\gamma}_j \|_{L^2}.
\]
Thus there exists a constant $M_0>0$, which is independent on $j$, such that 
\[
\int_0^1 |\ep \nabla U(\gamma_j)|\, dt  \le M_0 (1+ \| \dot{\gamma}_j \|_{L^2}^2).
\]
By Lemma \ref{lem:U} (ii), there exists $M_1>0$ such that
$U(q) \le  | \nabla U(q) |/4M_0 + M_1$ for any $q \in \interior Q$. Thus
\begin{align*}
\int_0^1 \ep U(\gamma_j) \, dt 
&\le \frac{1}{4M_0} \int_0^1 |\ep \nabla U(\gamma_j)| \, dt + \ep M_1
\le \frac{1+ \| \dot{\gamma}_j \|_{L^2}^2}{4}  + \ep M_1 \\
&= \frac{1}{2} \biggl( \mca{L}_{\ep U}(\gamma_j) + \int_0^1 \ep U(\gamma_j) \, dt \biggr) + \frac{1}{4} + \ep M_1. 
\end{align*}
Therefore, we obtain 
\[
\frac{1}{2} \int_0^1 \ep U(\gamma_j) \, dt \le
\frac{1}{2} \mca{L}_{\ep U}(\gamma_j) + \frac{1}{4} + \ep M_1.
\]
Since $\mca{L}_{\ep U}(\gamma_j)$ is bounded on $j$, we obtain: $\sup_j \int_0^1 \ep U(\gamma_j) \, dt < \infty$, $\sup_j \| \dot{\gamma}_j \|_{L^2}  < \infty$. 
Since $U(q)= h(q)^{-2} - (2\delta)^{-2}$, Lemma \ref{lem:boundary} implies $\inf_j \dist(\gamma_j([0,1]), \partial Q)>0$. 
\end{proof} 

\begin{lem}\label{lem:PS}
$\mca{L}_V$ satisfies the PS-condition. 
\end{lem}
\begin{proof}
Let $(\gamma_j)_j$ be a PS-sequence of $\mca{L}_V$. 
By Lemma \ref{lem:PS_boundary}, there exists a compact submanifold $Q' \subset \interior Q$, such that 
$\gamma_j([0,1]) \subset Q'$ for all $j$. 

It is easy to show that there exists a closed Riemannian manifold $M$, an isometric embedding $e: Q'  \to M$ and 
$W \in C^\infty([0,1] \times M)$, such that $V(t,q) = W(t,e(q))$ for all $t \in [0,1]$, $q \in Q'$. 
Then $\{e(\gamma_j)\}_j$ is a PS-sequence of $\mca{L}_W$, hence by Lemma \ref{lem:PS_closed}, it has a convergent subsequence. 
Thus $(\gamma_j)_j$ also has a convergent subsequence. 
\end{proof} 

We prove another consequence of Lemma \ref{lem:PS_boundary}. 
We define \textit{spectrum} of $\mca{L}_V$ as: 
\[
\Spec(\mca{L}_V):=\{ \mca{L}_V(\gamma) \mid d \mca{L}_V(\gamma)=0 \} \subset \R.
\]

\begin{lem}\label{lem:sard} 
$\Spec(\mca{L}_V)$ is closed in $\R$, and has a zero measure. 
\end{lem}
\begin{proof}
Closedness is immediate since $\mca{L}_V$ satisfies the PS-condition.
To show that $\Spec(\mca{L}_V)$ has a zero measure, we modify the arguments in \cite{Schwarz} pp.436.

For each $x \in \interior Q$, we define $\gamma_x: [0,1] \to \interior Q$ by $\gamma_x(0)=x$, $\dot{\gamma}_x(0)=0$ and $\ddot{\gamma}_x(t) + \nabla V_t(\gamma_x(t)) \equiv 0$.  
Then, $f(x):=\mca{L}_V(\gamma_x)$ is a $C^\infty$ function on $\interior Q$, and $\Spec(\mca{L}_V)$ is contained in the set of critical values of $f$. 
Hence our claim follows from the Sard theorem for finite dimensional manifolds. 
\end{proof} 
\begin{rem}
The above argument does not apply directly when we replace $\Lambda$ with $\Omega$. 
In other words, it is nontrivial to show that $\Spec(\mca{L}^\Omega_V)$ has a zero measure for $V \in C^\infty(S^1 \times TQ)$. 
One way to prove it is to apply Lemma 3.8 in \cite{Schwarz} directly to a Hamiltonian $H \in C^\infty(S^1 \times T^*Q)$, which is the Legendre transform of $V$. 
\end{rem} 

The following lemma is a key step in the proof. 

\begin{lem}\label{lem:AbM}
Let $c_-<c_+$ be real numbers such that $c_\pm \notin \Spec(\mca{L}_V)$. 
We set 
\[
\mca{C}_{(c_-,c_+)}:= \{ \gamma \in \Lambda \mid \mca{L}_V(\gamma) \in (c_-, c_+),\, d\mca{L}_V(\gamma)=0 \}. 
\]
If $H_j(\{\mca{L}_V<c_+ \}, \{\mca{L}_V<c_- \}) \ne 0$ and all elements of $\mca{C}_{(c_-,c_+)}$ are nondegenerate critical points of $\mca{L}_V$, 
there exists $\gamma \in \mca{C}_{(c_-, c_+)}$ such that $\ind(\gamma) = j$. 
\end{lem} 
\begin{proof} 
We use the theory developed in  \cite{AbM}, Section 2. 
Let us set 
\[
\hat{M}:= \{ \mca{L}_V < c_+\}, \quad
M:= \{c_- < \mca{L}_V < c_+\}, \quad 
f:= \mca{L}_V |_{\hat{M}}. 
\]
We take a smooth vector field $\hat{X}$ on $\hat{M}$, which is a negative scalar multiple of $\nabla f$ and satisfies the following properties:
\[
\| \nabla f(p) \|<1 \implies \hat{X}(p) = -\nabla f(p), \qquad
\| \nabla f(p) \| \ge 1 \implies  1 \le \| \hat{X}(p) \| \le 2.
\]
Let us examine whether
$\hat{M}$, $M$, $f$, $\hat{X}$ satisfy conditions (A1)-(A7) in \cite{AbM}, pp.22-23. 
(A1) follows from Lemma \ref{lem:complete}, (A6) follows from Lemma \ref{lem:PS}, and (A2)-(A5) are immediate. 
Since $\hat{X}$ is smooth, (A7) is also achieved by a small perturbation of $\hat{X}$, without violating (A1)-(A6)
(See Remark 2.1 \cite{AbS}).
Now our claim follows from Theorem 2.8 in \cite{AbM}. 
\end{proof} 

$V$ is said to be \textit{regular}, if all critical points of $\mca{L}_V$ are nondegenerate.  
The next Lemma \ref{lem:weber} shows that regularness can be achieved by compactly-supported small perturbations. 

\begin{lem}\label{lem:weber}
For any $V$ satisfying V-(i), there exists a sequence $(V_m)_{m=1,2,\ldots}$ such that all $V_m$ are regular and satisfy V-(i), and 
$(V_m)_m$ converges to $V$ in $C^\infty$-topology, i.e. for any $k \ge 0$ there holds $\lim_{m \to \infty} \| V- V_m \|_{C^k([0,1] \times Q)}=0$
(notice that $V- V_m$ extends to a $C^\infty$-function on $[0,1] \times Q$).
\end{lem} 

Lemma \ref{lem:weber} can be proved as Theorem 1.1 in \cite{Weber}. 
The setting in \cite{Weber} is a bit different from ours: in \cite{Weber}, the base Riemannian manifold is closed, and the Lagrangian is parametrized by $S^1$. 
However, these differences do not affect the proof. 
Now we can finish the proof of Proposition \ref{prop:Morsetheory} (i). 
As we explained at the beginning of this subsection, the proof of (ii) is parallel and omitted. 

\begin{proof} [\textbf{Proof of Proposition \ref{prop:Morsetheory}(i)}]
First we consider the case when $a, b \notin \Spec (\mca{L}_V)$. 
Since $\Spec (\mca{L}_V)$ is closed, there exists $c>0$ such that $[a-c,a+c]$ and $[b-c,b+c]$ are disjoint from $\Spec (\mca{L}_V)$. 
By Lemma \ref{lem:AbM}, inclusions 
\[
\{ \mca{L}_V < a-c \} \subset \{ \mca{L}_V < a\} \subset \{ \mca{L}_V < a+c \}, \quad
\{ \mca{L}_V < b-c \} \subset \{ \mca{L}_V < b\} \subset \{ \mca{L}_V < b+c \}
\]
induce isomorphisms on homologies. 
In particular, the homomorphism 
\[
H_j( \{ \mca{L}_V < b-c\}, \{ \mca{L}_V < a-c\}) \to 
H_j( \{ \mca{L}_V < b+c \}, \{ \mca{L}_V < a+c \})
\]
induced by inclusion is an isomorphism, and homologies on both sides are isomorphic 
to $H_j(\{ \mca{L}_V < b \}, \{ \mca{L}_V < a\})$, which is nonzero by our assumption. 

Take a sequence $(V_m)_m$ as in Lemma \ref{lem:weber}. 
For sufficiently large $m$, we have 
\[
\{ \mca{L}_V < a-c \} \subset \{ \mca{L}_{V_m} < a \} \subset \{ \mca{L}_V < a+c \}, \, 
\{ \mca{L}_V < b-c \} \subset \{ \mca{L}_{V_m} < b \} \subset \{ \mca{L}_V < b+c \}. 
\]
Hence there holds $H_j(\{\mca{L}_{V_m}<b\}, \{ \mca{L}_{V_m}<a\}) \ne 0$. 
By Lemma \ref{lem:AbM}, 
there exists $\gamma_m \in \Lambda$ such that 
$d\mca{L}_{V_m}(\gamma_m)=0$, 
$\mca{L}_{V_m}(\gamma_m) \in [a,b]$ and 
$\ind (\gamma_m) =j$. 

Since $\lim_{m \to \infty} \| V_m - V \|_{C^1} =0$, 
$(\gamma_m)_m$ is a PS-sequence of $\mca{L}_V$. 
Hence $(\gamma_m)_m$ has a convergent subsequence, and its limit $\gamma$ satisfies 
$d\mca{L}_V(\gamma)=0$, 
$\mca{L}_V(\gamma) \in [a,b]$. 
$\ind (\gamma) \le j$ follows from 
$\ind(\gamma) \le \liminf_{m \to \infty} \ind(\gamma_m)$, 
which easily follows from (\ref{eq:Hessian}).

Finally we consider the general case, i.e. $a$ and $b$ might be in $\Spec (\mca{L}_V)$. 
Since $\Spec (\mca{L}_V)$ has a zero measure, 
there exist increasing sequences $(a_m)_m$, $(b_m)_m$ such that $a_m,b_m \notin \Spec (\mca{L}_V)$ for every $m$, and 
$a=\lim_m a_m$, $b=\lim_m b_m$. 
Then, for sufficiently large $m$, $H_j(\{\mca{L}_V < b_m\}, \{\mca{L}_V < a_m\}) \ne 0$. 
Therefore there exists $\gamma_m$ such that $d\mca{L}_V(\gamma_m)=0$, 
$\mca{L}_V(\gamma_m) \in [a_m,b_m]$, $\ind(\gamma_m) \le j$. 
$(\gamma_m)_m$ is a PS-sequence of $\mca{L}_V$, thus it has a convergent subsequence. 
Then its limit $\gamma$ satisfies $d\mca{L}_V(\gamma)=0$, $\mca{L}_V(\gamma) \in [a,b]$, $\ind (\gamma) \le j$. 
\end{proof} 

\section{Billiard trajectory as a limit}

As in the previous section, we fix $\delta>0$ and use abbreviations $h:=h_\delta$, $U:=U_\delta$. 
The goal of this section is to prove Proposition \ref{prop:limit} below, which enables us to get a billiard trajectory as a limit of 
solutions of the approximating problem. 

\begin{prop}\label{prop:limit}
Let $a<b$ be positive real numbers, and $j$ be a nonnegative integer. 
\begin{enumerate}
\item[(i):] Suppose that for sufficiently small any $\ep>0$, there exists $\gamma_\ep \in \Lambda$ such that
$d\mca{L}^{\Lambda}_{\ep U}(\gamma_\ep) = 0$, $\mca{L}^{\Lambda}_{\ep U}(\gamma_\ep) \in [a,b]$ 
and $\ind (\gamma_\ep) \le j$. 
Then, there exists a brake billiard trajectory $\gamma$, such that 
$\sharp \mca{B}_\gamma  \le j-2$ and $\len(\gamma) \in [\sqrt{2a}, \sqrt{2b}]$. 
\item[(ii):] Suppose that for sufficiently small any $\ep>0$, there exists $\gamma_\ep \in \Omega$ such that
$d\mca{L}^{\Omega}_{\ep U}(\gamma_\ep)=0$, $\mca{L}^{\Omega}_{\ep U}(\gamma_\ep) \in [a,b]$, and $\ind (\gamma_\ep) \le j$. 
Then, there exists a periodic billiard trajectory $\gamma$, such that 
$\sharp \mca{B}_\gamma \le j$ and $\len(\gamma) \in [\sqrt{2a}, \sqrt{2b}]$. 
\end{enumerate}
\end{prop}

We only prove (i), since (ii) can be proved by parallel arguments. 
In the following arguments, we fix $\gamma_\ep$ for each $\ep$, 
as in Proposition \ref{prop:limit} (i). 
We abbreviate $\mca{L}^\Lambda_{\ep U}$ as $\mca{L}_\ep$. 

\begin{lem}\label{lem:limepU}
$\lim_{\ep \to 0} \int_0^1 \ep U(\gamma_\ep) \, dt =0$.
\end{lem}
\begin{proof}
Let us take $\nu$ as in Lemma \ref{lem:U} (i). 
By $\ddot{\gamma}_\ep + \ep \nabla U(\gamma_\ep) \equiv 0$ and $\dot{\gamma}_\ep(0) = \dot{\gamma}_\ep(1)=0$, we have 
\[
\int_0^1 \ep | \nabla U(\gamma_\ep) | \, dt
=\int_0^1 \langle \ep \nabla U(\gamma_\ep), \nu(\gamma_\ep) \rangle \, dt 
=\int_0^1 \langle \dot{\gamma_\ep}, \nabla_t(\nu(\gamma_\ep)) \rangle \, dt.
\]
Setting $M_0:= \max_{q \in Q} | \nabla \nu(q) |$, there holds 
\begin{equation}\label{eq:UdU}
\int_0^1 \ep |\nabla U(\gamma_\ep)| \, dt 
\le M_0 \| \dot{\gamma}_\ep \|_{L^2}^2 
= 2M_0\biggl( \mca{L}_\ep(\gamma_\ep) + \int_0^1 \ep U(\gamma_\ep) \, dt \biggr).
\end{equation}
By Lemma \ref{lem:U} (ii), 
there exists $M_1>0$ such that 
$U(q) \le |\nabla U(q)|/4M_0 + M_1$ for any $q \in \interior Q$. 
By same arguments as in the proof of Lemma \ref{lem:PS_boundary}, we get 
\[
\frac{1}{2} \int_0^1 \ep U(\gamma_\ep) \, dt \le \frac{1}{2}  \mca{L}_\ep(\gamma_\ep) + \ep M_1.
\]
Since $\sup_\ep \mca{L}_\ep (\gamma_\ep) \le b$, 
the above estimate implies $\sup_\ep \int_0^1 \ep U(\gamma_\ep) \, dt< \infty$. 
By (\ref{eq:UdU}), we get $\sup_\ep \int_0^1 \ep |\nabla U(\gamma_\ep)| \, dt < \infty$. 
The following identity is clear from the definition of $U$:
\[
\int_0^1 \ep |\nabla U(\gamma_\ep)| \, dt 
=\int_0^1 2 \ep |\nabla h(\gamma_\ep)| h(\gamma_\ep)^{-3} \, dt. 
\]
Since $|\nabla h(q)|=1$ for any $q$ such that $h(q) < \delta$, we get 
$\sup_\ep \int_0^1 \ep h(\gamma_\ep)^{-3} \, dt < \infty$. 
Finally, by H\"{o}lder inequality we get
\[
\limsup_{\ep \to 0}  \int_0^1 \ep h(\gamma_\ep)^{-2} \, dt 
\le \limsup_{\ep \to 0} \biggl(\int_0^1 \ep h(\gamma_\ep)^{-3} \, dt\biggr)^{2/3} \cdot \ep^{1/3} =0.
\]
Since $0 \le \ep U(q) \le \ep h(q)^{-2}$ for any $q \in \interior Q$, we obtain 
$\lim_{\ep \to 0} \int_0^1 \ep U(\gamma_\ep) \, dt =0$. 
\end{proof}

\begin{cor}\label{cor:limepU}
The following quantities are bounded on $\ep$: 
\[
\int_0^1 |\ddot{\gamma}_\ep(t)|\,dt = \int_0^1 \ep |\nabla U(\gamma_\ep)| \, dt, \quad
\int_0^1 \ep h(\gamma_\ep)^{-3} \, dt, \quad
E(\gamma_\ep):=|\dot{\gamma}_\ep|^2/2 + \ep U(\gamma_\ep).
\]
\end{cor}
\begin{proof}
In the course of the proof of Lemma \ref{lem:limepU}, we have shown that the first two quantities are bounded. 
$\sup_\ep E(\gamma_\ep) < \infty$ follows from the identity 
\[
E(\gamma_\ep) = \int_0^1 \frac{|\dot{\gamma}_\ep|^2}{2} + \ep U(\gamma_\ep) \, dt = \mca{L}_\ep(\gamma_\ep) + 2  \int_0^1 \ep U(\gamma_\ep) \, dt, 
\]
$\sup_\ep \mca{L}_\ep(\gamma_\ep) \le b$ and $\sup_\ep \int_0^1 \ep U(\gamma_\ep) \, dt < \infty$. 
\end{proof}

By Corollary \ref{cor:limepU}, $\ddot{\gamma}_\ep$ is $L^1$-bounded. 
Since $W^{2,1}([0,1])$ is compactly embedded to $W^{1,2}([0,1])$, a certain subsequence of $(\gamma_\ep)_\ep$ is convergent  in $W^{1,2}([0,1],Q)$ as $\ep \to 0$. 
We denote the limit as $\gamma_0$. 
Moreover, since $2 \ep h(\gamma_\ep)^{-3}$ is $L^1$-bounded, up to subsequence it 
converges to a certain Borel measure $\mu \ge 0$ on $[0,1]$ in weak sense, i.e. for any $f \in C^0([0,1])$ there holds
\[
\lim_{\ep \to 0} 
\int_0^1 f(t) 2 \ep h(\gamma_\ep(t))^{-3} \,dt  =  \int_0^1 f(t) \, d \mu(t) .
\]

For any $t \in [0,1]$ and $c>0$, we set $B_c(t):= \{ s \in [0,1] \mid |s-t| < c\}$. 
The \textit{support} of $\mu$ is defined as
$\supp \mu:=  \{ t \in [0,1] \mid \forall c>0, \mu(B_c(t)) >0 \}$.

\begin{lem}\label{lem:supp_mu}
There holds 
$\supp \mu \subset \gamma_0 ^{-1}(\partial Q)$ and $\sharp \supp \mu \le j$. 
\end{lem}
\begin{proof}
If $\tau \in [0,1]$ satisfies $\gamma_0(\tau) \notin \partial Q$, 
$\ep h(\gamma_\ep(t))^{-3}$ converges uniformly to $0$ in a neighborhood of $\tau$, thus $\tau \notin \supp \mu$. 
Therefore $\supp \mu \subset \gamma_0^{-1}(\partial Q)$. 

We show that $\sharp \supp \mu \le j$. 
For any $\tau \in \supp \mu$, we have shown that $\gamma_0(\tau) \in \partial Q$, hence $d(\gamma_0(\tau)) =0$. 
We take $c>0$ so that $d(\gamma_0(t)) < \delta$ for any $t \in B_c(\tau)$. 

We take $\psi \in C^\infty([0,1])$ so that $0 \le \psi(t) \le 1$ for any $t$, 
$\supp \psi \subset B_c(\tau)$, and $\psi \equiv 1$ on $B_{c/2}(\tau)$.  
Let $v_\ep(t):= \psi(t) \nabla h(\gamma_\ep(t))$. 
Our aim is to show 
\begin{equation}\label{eq:Hessian_infty} 
\lim_{\ep \to 0} d^2 \mca{L}_\ep (\gamma_\ep) (v_\ep, v_\ep) = -\infty.
\end{equation}
Obviously $\supp v_\ep \subset B_c(\tau)$, and we may take $c>0$ arbitrarily small. 
Hence once we prove (\ref{eq:Hessian_infty}), it is easy to show that 
$\liminf_{\ep \to 0} \ind (\gamma_\ep) \ge \sharp \supp \mu$. 
On the other hand, by our assumption $\ind(\gamma_\ep) \le j$ for any $\ep>0$. 
Hence $\sharp \supp \mu \le j$. 

Now we show (\ref{eq:Hessian_infty}). By (\ref{eq:Hessian}), there holds 
\begin{align*}
d^2 \mca{L}_\ep (\gamma_\ep)(v_\ep, v_\ep)& = \int_0^1 |\nabla_t v_\ep|^2 - \langle R(\dot{\gamma}_\ep, v_\ep)(v_\ep), \dot{\gamma}_\ep \rangle \,  dt \\
&+ 2\ep \int_0^1 \langle \nabla_{v_\ep} \nabla h(\gamma_\ep), v_\ep \rangle h(\gamma_\ep)^{-3} \, dt 
- 6\ep \int_0^1 \{dh(\gamma_\ep)(v_\ep) \}^2 h(\gamma_\ep)^{-4} \, dt.
\end{align*}
By Corollary \ref{cor:limepU}, $\sup_\ep \| \dot{\gamma}_\ep \|_{L^\infty} < \infty$. 
Thus it is easy to check that the first integral is bounded on $\ep$. 
Corollary \ref{cor:limepU} also shows $\sup_\ep \int_0^1 \ep h(\gamma_\ep)^{-3} \,dt < \infty$, thus the 
second integral is bounded on $\ep$. 

Recall that $d(\gamma_0(t)) < \delta$ for any $t \in B_c(\tau)$. 
Hence when $\ep>0$ is sufficiently small, $d(\gamma_\ep(t)) < \delta$ for any $t \in B_{c/2}(\tau)$. 
For such $\ep>0$, 
$dh(\gamma_\ep)(v_\ep) = |\nabla h(\gamma_\ep)|^2 = 1$ on $B_{c/2}(\tau)$. 
Therefore 
\[
\ep \int_0^1 \{ dh(\gamma_\ep) (v_\ep) \}^2 h(\gamma_\ep)^{-4} \, dt 
\ge 
\ep \int_{B_{c/2}(\tau)} h(\gamma_\ep)^{-4} \, dt 
\ge (c\ep)^{-1/3} \biggr(\int_{B_{c/2}(\tau)} \ep h(\gamma_\ep)^{-3} \, dt \biggr)^{4/3}. 
\]
The second inequality follows from H\"{o}lder.
Since $\tau \in \supp \mu$, 
\[ 
\liminf_{\ep \to 0} \int_{B_{c/2}(\tau)} \ep h(\gamma_\ep)^{-3} \, dt  \ge  \mu(B_{c/2}(\tau))/2 >0.
\]
Hence $\lim_{\ep \to 0} \ep \int_0^1 \{ dh(\gamma_\ep) (v_\ep) \}^2 h(\gamma_\ep)^{-4} \, dt = \infty$, 
thus we have proved (\ref{eq:Hessian_infty}).
\end{proof}

For $q \in \partial Q$, let $\nu(q)$ denote the unit vector which is outer normal to $\partial Q$ at $q$. 

\begin{lem}\label{lem:weakbilliard}
For any $v \in W^{1,2}([0,1], \gamma_0^*(TQ))$, there holds
\[
\int_0^1 \langle \dot{\gamma}_0, \nabla_t v \rangle \, dt  = \int_0^1 \langle \nu(\gamma_0), v  \rangle \, d \mu(t). 
\]
Notice that RHS is well-defined, since $\supp \mu \subset \gamma_0^{-1}(\partial Q)$.
\end{lem}
\begin{proof}
One can take $v_\ep \in T_{\gamma_\ep} \Lambda$ so that $v_\ep \to v$ as $\ep \to 0$, in $W^{1,2}$-norm. 
By $\ddot{\gamma}_\ep + \ep \nabla U(\gamma_\ep) \equiv 0$ and $\dot{\gamma}_\ep(0) = \dot{\gamma}_\ep(1)=0$, we get 
\[
\int_0^1 \langle \ep \nabla U(\gamma_\ep), v_\ep(t) \rangle \, dt = - \int_0^1 \langle \ddot{\gamma}_\ep(t), v_\ep(t) \rangle \,dt = \int_0^1 \langle \dot{\gamma}_\ep(t), \nabla_t(v_\ep(t)) \rangle \, dt.
\]
As $\ep \to 0$, RHS goes to $\int_0^1 \langle \dot{\gamma}_0, \nabla_t v \rangle \, dt$. 
On the other hand, since $\nabla U(q) = -2 \nabla h(q) h(q)^{-3}$, LHS goes to $\int_0^1  \langle \nu(\gamma_0 ), v \rangle \, d \mu(t)$ as $\ep \to 0$. 
\end{proof} 

Lemma \ref{lem:weakbilliard} shows that $\ddot{\gamma}_0 \equiv 0$ on $[0,1] \setminus \supp \mu$. 
Lemma \ref{lem:supp_mu} shows that $\supp \mu$ is discrete. 
Hence $\dot{\gamma}_0^-(t)= \lim_{h \to 0-} \dot{\gamma}_0(t+h)$ exists for any $t>0$, and 
$\dot{\gamma}_0^+(t)=\lim_{h \to 0+} \dot{\gamma}_0(t+h)$ exists for any $t<1$.
Now we show that $\gamma_0$ satisfies the following properties: 
\begin{itemize}
\item $\len (\gamma_0) \in [\sqrt{2a}, \sqrt{2b}]$. 
\item $\{0,1\} \subset \supp \mu$. Moreover, $\dot{\gamma}_0(0)$, $\dot{\gamma}_0(1)$ are perpendicular to $\partial Q$. 
\item $\gamma_0$ satisfies the law of reflection at every point on $\supp \mu \setminus \{0,1\}$. 
\end{itemize}
Once these properties are confirmed, $\gamma_0$ is a brake billiard trajectory with $\mca{B}_{\gamma_0} = \supp \mu \setminus \{0,1\}$, 
and Proposition \ref{prop:limit} (i) is proved.

Let $I$ be any interval on $[0,1]$. By Lemma \ref{lem:limepU}, 
\[
\int_I  | \dot{\gamma}_0|^2 \, dt = \lim_{\ep \to 0} \int_I | \dot{\gamma}_\ep|^2 \, dt = \lim_{\ep \to 0} 2\biggl(|I| E(\gamma_\ep) - \int_I \ep U(\gamma_\ep) \, dt \biggr) = 2 |I| \lim_{\ep \to 0} E(\gamma_\ep).
\]
Hence $E:= \lim_{\ep \to 0} E(\gamma_\ep)$ exists, and there holds $|\dot{\gamma}_0(t)| = \sqrt{2E}$ for any $t \notin \supp \mu$. 
Then, $\len (\gamma_0) \in [\sqrt{2a}, \sqrt{2b}]$ follows from: 
\[
E=\lim_{\ep \to 0} E(\gamma_\ep) = \lim_{\ep \to 0} \mca{L}_\ep(\gamma_\ep) + 2 \int_0^1 \ep U(\gamma_\ep) \, dt = \lim_{\ep \to 0} \mca{L}_\ep(\gamma_\ep)  \in [a,b].
\]

Let us prove that $0 \in \supp \mu$. 
If not, there exists $c>0$ such that $\mu \equiv 0$ on $[0,c]$. 
Take $f \in C^\infty([0,1])$ such that $f(0)=1$ and $\supp f \subset [0,c]$. 
Let $v(t):= f(t) \dot{\gamma}_0(t)$. 
Then, Lemma \ref{lem:weakbilliard} implies 
\[
0 = \int_0^1 \langle \dot{\gamma}_0, \nabla_t v \rangle \, dt = \int_0^1 f'(t) |\dot{\gamma}_0(t)|^2 \, dt = -2E. 
\]
This contradicts $E \in [a,b]$ and $a>0$, 
hence $0 \in \supp \mu$. We can show $1 \in \supp \mu$ by same arguments. 

Let us prove that $\dot{\gamma}_0(0)$ is perpendicular to $\partial Q$.
Let $\zeta_0$ be any tangent vector of $\partial Q$ at $\gamma_0(0)$.
Take $c>0$ sufficiently small so that $[0,c] \cap \supp \mu = \{0\}$, and define $\zeta(t) \in T_{\gamma_0(t)}Q$ for any $0 \le t \le c$ by  
$\zeta(0)=\zeta_0$, and $\nabla_t \zeta \equiv 0$. Take $f \in C^\infty([0,1])$ as above, and set $v(t):= f(t) \zeta(t)$. 
Then Lemma \ref{lem:weakbilliard} implies 
\[
\mu(\{0\}) \langle \nu(\gamma_0(0)), \zeta_0 \rangle = \int_0^1 \langle \nu(\gamma_0), v  \rangle \, d\mu(t)  = \int_0^1 \langle \dot{\gamma}_0, \nabla_t v \rangle dt = - \langle \zeta_0, \dot{\gamma}_0(0) \rangle. 
\]
Since $\zeta_0$ is tangent to $\partial Q$, LHS is zero, thus $\langle \zeta_0, \dot{\gamma}_0(0) \rangle =0$. 
This shows that $\dot{\gamma}_0(0)$ is perpendicular to $\partial Q$. 

Finally, let us prove that $\gamma_0$ satisfies the law of reflection at any $t \in \supp \mu \setminus \{0,1\}$. 
Similar arguments as above show that $\dot{\gamma}_0^+(t) - \dot{\gamma}_0^-(t)$ is nonzero and perpendicular to $\partial Q$. 
On the other hand, $|\dot{\gamma}_0^+(t)|=|\dot{\gamma}_0^-(t)|$, since both are equal to $\sqrt{2E}$. 
Then it is immediate that $\gamma_0$ satisfies the law of reflection at $t$.

We have now finished the proof of Proposition \ref{prop:limit} (i). 
As we explained at the beginning of this section, (ii) can be proved by parallel arguments. 

\section{Proof of Theorem \ref{mainthm}}
In this section, we complete the proof of Theorem \ref{mainthm}. 
We only prove (i), since (ii) can be proved by parallel arguments. 
We may assume that $Q$ is connected and $\partial Q \ne \emptyset$ (see Remark \ref{rem:closed}). 
First we need the following technical lemma. Let us denote
\[
\Lambda^c(\interior Q):= \Lambda^c(Q) \cap \Lambda(\interior Q), \qquad
\Lambda_\delta(\interior Q):= \Lambda_\delta(Q) \cap \Lambda(\interior Q).
\]

\begin{lem}\label{lem:QandintQ}
For any $c \in \R$ and $\delta>0$, there holds 
\[
H_*(\Lambda^c(Q) \cup \Lambda_\delta(Q), \Lambda^c(\interior Q) \cup \Lambda_\delta(\interior Q)) =0.
\]
\end{lem}
\begin{proof}
It is enough to show that the inclusion 
\begin{equation}\label{eq:QandintQ}
\Lambda^c(\interior Q) \cup \Lambda_\delta(\interior Q) \to
\Lambda^c(Q) \cup \Lambda_\delta(Q)
\end{equation}
is a homotopy equivalence. 
Let $Z$ be a smooth vector field on $Q$, which points strictly inwards on $\partial Q$, and $Z \equiv 0$ on $Q(\delta)$. 
Let $(\psi^t)_{t \ge 0}$ be the isotopy generated by $Z$, i.e. $\psi^0=\id_Q$, $\partial_t \psi^t = Z(\psi^t)$. 
Then, it is easy to show that
\[
\Lambda^c(Q) \cup \Lambda_\delta(Q) \to 
\Lambda^c(\interior Q) \cup \Lambda_\delta(\interior Q): \quad \gamma \mapsto \psi^1 \circ \gamma 
\]
is a homotopy inverse of (\ref{eq:QandintQ}). 
\end{proof} 

By Lemma \ref{lem:QandintQ}, the assumption of Theorem \ref{mainthm} (i) is equivalent to 
\[
\plim_{\delta \to 0} H_j( \Lambda^b(\interior Q) \cup \Lambda_\delta (\interior Q), \Lambda^a (\interior Q) \cup \Lambda_\delta (\interior Q)) \ne 0.
\]
In this section, we abbreviate $\Lambda^b(\interior Q)$ by $\Lambda^b$, $\Lambda_\delta(\interior Q)$ by $\Lambda_\delta$, and so on. 
There exists $\delta_0>0$ such that 
\begin{equation}\label{eq:limdeltatodelta0}
\plim_{\delta \to 0} H_j(\Lambda^b \cup \Lambda_\delta, \Lambda^a \cup \Lambda_\delta) \to
H_j(\Lambda^b \cup \Lambda_{\delta_0}, \Lambda^a \cup \Lambda_{\delta_0})
\end{equation}
is nonzero. We take $\delta_1>0$ so that $3\delta_1 \le \delta_0$. 
We are going to prove 
\[
H_j( \{ \mca{L}^\Lambda_{\ep U_{\delta_1} }<b\}, \{ \mca{L}^\Lambda_{\ep U_{\delta_1}} < a \}) \ne 0
\]
for any $\ep>0$. Once we prove this, 
Proposition \ref{prop:Morsetheory} and 
Proposition \ref{prop:limit} show that there exists a brake billiard trajectory $\gamma$ such that $\sharp \mca{B}_\gamma \le j-2$ and 
$\len(\gamma) \in [\sqrt{2a}, \sqrt{2b}]$.

We fix $\ep>0$. 
For any $c>0$ there holds 
$\{ \mca{L}^\Lambda_{\ep U_{\delta_1}} < c\}  \subset \Lambda^c \cup \Lambda_{\delta_0}$, since $U_{\delta_1} \equiv 0$ on $Q(\delta_0)$. 
On the other hand, Lemma \ref{lem:boundary} shows that, for sufficiently small $\delta_2>0$ there holds 
$\Lambda^b \cap \Lambda_{\delta_2} \subset \{ \mca{L}^\Lambda_{\ep U_{\delta_1}} < a \}$. 
Thus we have the following commutative diagram, all homomorphisms are induced by inclusions: 
\[
\xymatrix{
H_j (\Lambda^b , \Lambda^a \cup (\Lambda^b \cap \Lambda_{\delta_2})) \ar[r]\ar[d] & H_j( \{ \mca{L}^\Lambda_{\ep U_{\delta_1}} <b\}, \{ \mca{L}^\Lambda_{\ep U_{\delta_1}} < a\} ) \ar[d] \\
H_j (\Lambda^b \cup \Lambda_{\delta_2}, \Lambda^a \cup \Lambda_{\delta_2} ) \ar[r] & H_j(\Lambda^b \cup \Lambda_{\delta_0}, \Lambda^a \cup \Lambda_{\delta_0}).
}
\]
Since (\ref{eq:limdeltatodelta0}) is nonzero, the bottom arrow is nonzero. 
On the other hand, the excision property shows that the left vertical arrow is an isomorphism.
By commutativity of the diagram, we have 
$H_j( \{ \mca{L}^\Lambda_{\ep U_{\delta_1}} <b\}, \{ \mca{L}^\Lambda_{\ep U_{\delta_1}} < a\} )  \ne 0$, 
and this completes the proof.

\section{Short billiard trajectory} 

In this section, we prove Theorem \ref{thm:short}. 
In Section 5.1, we introduce the notion of capacity for Riemannian manifolds with boundaries, and 
show that the capacity detects the length of a billiard trajectory
(Lemma \ref{lem:spectrality}). 
In Section 5.2, we bound the capacity by the inradius, and complete the proof of Theorem \ref{thm:short}. 
In Section 5.3, we prove Corollary \ref{cor:gdloop} as a consequence of Theorem \ref{thm:short}. 

\subsection{Capacity} 
First we introduce some notations. 
\begin{itemize}
\item We define $\Lambda_\partial (Q) \subset \Lambda(Q)$, $\Omega_\partial (Q) \subset \Omega(Q)$ by 
\[
\Lambda_\partial(Q):= \{ \gamma \in \Lambda(Q) \mid \gamma([0,1]) \cap \partial Q \ne \emptyset \}, \quad
\Omega_\partial(Q):=  \{ \gamma \in \Omega(Q) \mid \gamma(S^1) \cap \partial Q \ne \emptyset \}.
\]
\item For each $q \in Q$, $p_q$ denotes the constant path at $q$, and $l_q$ denotes the constant loop at $q$. 
\end{itemize} 
We often identify $q \in Q$ with $p_q$ and $l_q$, thus have inclusions $Q \to \Lambda(Q)$, $Q \to \Omega(Q)$. 
For each $a>0$, we consider the following homomorphisms, all induced by inclusions. 
\begin{align*}
I^{\Lambda,a}_0&: H_*(Q,\partial Q) \to H_*(\Lambda^a (Q) \cup \Lambda_\partial (Q), \Lambda_\partial (Q)), \\
I^{\Lambda,a}_1&: H_*(Q,\partial Q) \to \plim_{\delta \to 0} H_*(\Lambda^a(Q) \cup \Lambda_\delta(Q), \Lambda_\delta(Q)), \\
I^{\Lambda,a}_2&: H_*(Q,\partial Q) \cong \plim_{\delta \to 0} H_*(\interior Q, \interior Q \setminus Q(\delta)) \to \plim_{\delta \to 0} H_*(\Lambda^a(\interior Q) \cup \Lambda_\delta(\interior Q), \Lambda_\delta(\interior Q)).
\end{align*}
One can define $I^{\Omega,a}_0$, $I^{\Omega,a}_1$, $I^{\Omega,a}_2$ in the same manner. 

\begin{lem}\label{lem:capacity} 
For any $\alpha \in H_*(Q,\partial Q)$ and $j=0,1,2$, let us define 
\[
c^\Lambda_j(\alpha):= \inf \{ c > 0 \mid I^{\Lambda,c^2/2}_j(\alpha)=0 \}.
\]
Then $c^\Lambda_0(\alpha)=c^\Lambda_1(\alpha)=c^\Lambda_2(\alpha)$. 
\end{lem} 
\begin{proof}
$c^\Lambda_1(\alpha)= c^\Lambda_2(\alpha)$ is immediate from Lemma \ref{lem:QandintQ}.
$c^\Lambda_1(\alpha) \le c^\Lambda_0(\alpha)$ is also clear, since there exists a natural homomorphism 
\[
H_*(\Lambda^a(Q) \cup \Lambda_\partial (Q), \Lambda_\partial (Q)) \to \plim_{\delta \to 0} H_*(\Lambda^a(Q) \cup \Lambda_\delta(Q), \Lambda_\delta(Q)), 
\]
which is induced by inclusions. 
Hence it is enough to prove $c^\Lambda_1(\alpha) \ge c^\Lambda_0(\alpha)$. 

Let $a>a'$ be any positive real numbers. 
When $\delta>0$ is sufficiently small, 
there exists a $C^\infty$ map $\psi: Q \times [0,1] \to Q; (x,t) \mapsto \psi_t(x)$ such that: 
\begin{itemize}
\item $\psi_0 = \id_Q$. $\psi_t|_{\partial Q} = \id_{\partial Q}$ for any $0 \le t \le 1$.  
\item $\psi_1 ( Q \setminus Q(\delta)) = \partial Q$. 
\item $|d\psi_1(\xi)| \le \sqrt{a/a'} |\xi|$ for any $\xi \in TQ$. 
\end{itemize} 
Then we have the following commutative diagram:
\[
\xymatrix{ 
H_*(Q,\partial Q) \ar[rr]\ar[rd]_{I^{\Lambda,a}_0}&& H_*(\Lambda^{a'}(Q) \cup \Lambda_\delta(Q), \Lambda_\delta(Q)) \ar[ld]^{(\psi_1)_*} \\
&H_*(\Lambda^a(Q) \cup \Lambda_\partial (Q) , \Lambda_\partial (Q)).&
}
\]
If $a'>c^\Lambda_1(\alpha)^2/2$, $\alpha \in H_*(Q,\partial Q)$ vanishes by the top arrow, hence $I^{\Lambda,a}_0(\alpha)=0$, 
thus $a>c^\Lambda_0(\alpha)^2/2$. 
Since we may take $a>a'$ arbitrarily, we have shown that $c^\Lambda_1(\alpha) \ge c^\Lambda_0(\alpha)$. 
\end{proof} 

For any $\alpha \in H_*(Q,\partial Q)$, we denote 
$c^\Lambda_0(\alpha)=c^\Lambda_1(\alpha)=c^\Lambda_2(\alpha)$ in Lemma \ref{lem:capacity} by $c^\Lambda(Q:\alpha)$. 
On the other hand, for $j=0,1,2$, we define 
\[
c^\Omega_j(\alpha):= \inf \{ c>0  \mid I^{\Omega,c^2/2}_j(\alpha)=0 \}.
\] 
By same arguments as in Lemma \ref{lem:capacity}, we can show that $c^\Omega_0(\alpha)=c^\Omega_1(\alpha)=c^\Omega_2(\alpha)$. 
We denote it by $c^\Omega(Q:\alpha)$. 
We call $c^\Lambda(Q:\alpha)$ and $c^\Omega(Q:\alpha)$ \textit{capacities} of $Q$. 

\begin{rem} 
The above definition of $c^\Lambda$ and $c^\Omega$ imitate the definition of the Floer-Hofer-Wysocki (FHW) capacity, which is due to \cite{FHW} (see also \cite{Ir}, Section 2.4). 
As a matter of fact, when $Q$ is a domain in the Euclidean space, $c^\Omega(Q: [Q,\partial Q])$ is equal to the 
FHW capacity of its disc cotangent bundle. 
See Corollary 1.4 in \cite{Ir}.
\end{rem} 

\begin{lem}\label{lem:nonvanishing}
For any $\alpha \in H_*(Q,\partial Q) \setminus \{0\}$, 
$c^\Lambda(Q:\alpha), c^\Omega(Q:\alpha)>0$. 
\end{lem}
\begin{proof}
We only prove $c^\Lambda(Q:\alpha)>0$, since $c^\Omega(Q:\alpha)>0$ can be proved by parallel arguments. 
In this proof, we use abbreviations $\Lambda^a:= \Lambda^a(\interior Q)$, $\Lambda_\delta:= \Lambda_\delta(\interior Q)$. 
For any positive $a$ and $\delta$, the excision property shows that 
\[
H_*(\Lambda^a, \Lambda^a \cap \Lambda_\delta ) \to 
H_*(\Lambda^a  \cup \Lambda_\delta, \Lambda_\delta)
\]
is an isomorphism. Therefore, it is enough to show that for sufficiently small $a>0$ 
\[
\plim_{\delta \to 0} H_*(\interior Q, \interior Q \setminus Q(\delta)) \to 
\plim_{\delta \to 0} H_*(\Lambda^a, \Lambda^a \cap \Lambda_\delta)
\]
is injective. 
For any $\gamma \in \Lambda^a \cap \Lambda_\delta$, there holds 
\[
\gamma(0) \in \interior Q \setminus Q(\delta + \len(\gamma)) \subset \interior Q \setminus Q(\delta + \sqrt{2a}).
\]
Define $\ev:\Lambda^a  \to \interior Q$ by $\ev(\gamma):=\gamma(0)$, and consider the commutative diagram 
\[
\xymatrix{
\plim_{\delta \to 0} H_*(\interior Q, \interior Q \setminus Q(\delta)) \ar[r] \ar[rd] &\plim_{\delta \to 0} H_*(\Lambda^a, \Lambda^a \cap \Lambda_\delta) \ar[d]^-{(\ev)_*} \\
& \plim_{\delta \to 0} H_*( \interior Q, \interior Q \setminus Q(\delta + \sqrt{2a})).}
\]
When $a>0$ is sufficiently small, the diagonal arrow is an isomorphism. 
Therefore the horizontal arrow is injective. 
\end{proof}

The next lemma shows that the capacity detects the length of a billiard trajectory. 

\begin{lem}\label{lem:spectrality} 
Let $\alpha \in H_j(Q,\partial Q) \setminus \{0\}$. 
\begin{enumerate} 
\item[(i):]  If $c^\Lambda(Q:\alpha)<\infty$, there exists a brake billiard trajectory $\gamma$ on $Q$ such that 
$\sharp \mca{B}_\gamma \le j-1$, $\len(\gamma)=c^\Lambda(Q:\alpha)$.
\item[(ii):]  If $c^\Omega(Q:\alpha)<\infty$, there exists a periodic billiard trajectory $\gamma$ on $Q$ such that 
$\sharp \mca{B}_\gamma \le j+1$, $\len(\gamma)=c^\Omega(Q:\alpha)$. 
\end{enumerate}
\end{lem}
\begin{proof}
We only prove (i), since (ii) can be proved by parallel arguments.
We set $a:= c^\Lambda(Q:\alpha)^2/2$. 
Then, for any $\ep>0$ there holds
$I^{\Lambda,a-\ep}_1(\alpha) \ne 0$ and 
$I^{\Lambda,a+\ep}_0(\alpha)=0$. 
In this proof, we use abbreviations 
$\Lambda^a:=\Lambda^a(Q)$, 
$\Lambda_\delta:=\Lambda_\delta(Q)$, 
$\Lambda_\partial:= \Lambda_\partial(Q)$ and so on. 

For any $\delta>0$, we have a commutative diagram
\[
\xymatrix{
H_{j+1}(\Lambda^{a+\ep} \cup \Lambda_\partial, \Lambda^{a-\ep} \cup \Lambda_\partial) \ar[r]^-{\partial_0}\ar[d]& H_j(\Lambda^{a-\ep} \cup \Lambda_\partial, \Lambda_\partial)\ar[d] \\
H_{j+1}(\Lambda^{a+\ep} \cup \Lambda_{\delta}, \Lambda^{a-\ep} \cup \Lambda_\delta  ) \ar[r]_-{\partial_\delta}& H_j(\Lambda^{a-\ep} \cup \Lambda_\delta, \Lambda_\delta), 
}
\]
where vertical arrows are induced by inclusions, and horizontal arrows are connecting homomorphisms. 
Since $I_0^{\Lambda,a+\ep}(\alpha) =0$, we have 
$I_0^{\Lambda,a-\ep}(\alpha) \in \Image \partial_0$. 
Letting $\delta \to 0$ of the above diagram, we have the following commutative diagram: 
\[
\xymatrix{
H_{j+1}(\Lambda^{a+\ep} \cup \Lambda_\partial, \Lambda^{a-\ep} \cup \Lambda_\partial) \ar[r]^-{\partial_0}\ar[d]& H_j(\Lambda^{a-\ep} \cup \Lambda_\partial, \Lambda_\partial)\ar[d]^{\iota} \\
\plim_{\delta \to 0} H_{j+1}(\Lambda^{a+\ep} \cup \Lambda_{\delta}, \Lambda^{a-\ep} \cup \Lambda_\delta  ) \ar[r] & \plim_{\delta \to 0} H_j(\Lambda^{a-\ep} \cup \Lambda_\delta, \Lambda_\delta). 
}
\]
Let us denote the right vertical arrow as $\iota$. 
Then, $\iota(I_0^{\Lambda,a-\ep}(\alpha)) = I_1^{\Lambda,a-\ep}(\alpha) \ne 0$. 
Since $I_0^{\Lambda,a-\ep}(\alpha) \in \Image \partial_0$, we get 
$\plim_{\delta \to 0} H_{j+1}(\Lambda^{a+\ep} \cup \Lambda_{\delta}, \Lambda^{a-\ep} \cup \Lambda_\delta  )  \ne 0$. 
By Theorem \ref{mainthm}, there exists a brake billiard trajectory $\gamma_\ep$ on $Q$ such that
$\sharp \mca{B}_{\gamma_\ep} \le j-1$ and 
$\len(\gamma_\ep) \in [\sqrt{2(a-\ep)}, \sqrt{2(a+\ep)}]$. 
As $\ep \to 0$, a certain subsequence of $(\gamma_\ep)_\ep$ converges to a brake billiard trajectory $\gamma$ such that 
$\sharp \mca{B}_\gamma \le j-1$, 
$\len (\gamma) = \sqrt{2a}= c^\Lambda(Q:\alpha)$. 
\end{proof}

\subsection{Capacity and inradius} 

By Lemma \ref{lem:spectrality}, 
Theorem \ref{thm:short} follows at once from the following proposition. 
Recall that $r(Q)$ denotes the inradius of $Q$. 

\begin{prop}\label{prop:bound}
Let $Q$ be a compact, connected Riemannian manifold with nonempty boundary, and $\alpha \in H_j(Q,\partial Q)$. 
Then, there holds $c^\Lambda(Q:\alpha) \le 2j r(Q)$, $c^\Omega(Q:\alpha) \le 2(j+1)r(Q)$. 
\end{prop}

The goal of this subsection is to prove Proposition \ref{prop:bound}. 
We will give a proof which stems from arguments in our paper \cite{Ir}, Section 7. 
First we need some preliminary results: 
Lemma \ref{lem:polyhedron}, Lemma \ref{lem:shortpath}. 

Let $P$ be a finite simplicial complex and $\sigma$ be a simplex on $P$. 
$\Star(\sigma) \subset P$ denotes the union of interiors of all simplices of $P$, which contain $\sigma$ as a facet, i.e.
$\Star(\sigma):= \bigcup_{\sigma \subset \tau} \interior \tau$.  

\begin{lem}\label{lem:polyhedron}
Let $P$ be a finite simplicial complex. 
There exist continuous functions $w_\sigma:P \to [0,1]$ where $\sigma$ runs over all simplices of $P$, 
such that the following holds: 
\begin{enumerate}
\item[(i):] For any simplex $\sigma$, $\supp w_\sigma \subset \Star(\sigma)$. 
\item[(ii):] For any distinct simplices $\sigma$, $\sigma'$ of same dimensions, $\supp w_\sigma \cap \supp w_{\sigma'} = \emptyset$.
\item[(iii):] $\bigcup_{\sigma} w_\sigma^{-1}(1) = P$, where $\sigma$ runs over all simplices of $P$. 
\end{enumerate} 
\end{lem}
\begin{proof}
The proof is by induction on $\dim P$. 
The claim is obvious when $\dim P=0$. 
Suppose that we have proved the claim for finite simplicial complexes of dimension $\le d-1$, and let $P$ be a finite simplicial complex of dimension $d$. 

Let $\sigma_1,\ldots,\sigma_m$ be the all simplices on $P$ of dimension $d$, and 
$P^{(d-1)}$ denote the union of all simplices on $P$ of dimension $\le d-1$. 
Take $x_j \in \interior \sigma_j$ for every $j=1,\ldots, m$.
There exists a continuous retraction $r: P \setminus \{x_1,\ldots,x_m\} \to P^{(d-1)}$ such that, 
there holds $r(\sigma_j \setminus \{x_j\}) = \partial \sigma_j$ for any $j=1,\ldots, m$. 

We define a continuous function $\tilde{w}_\sigma: P \to [0,1]$ for each simplex $\sigma$ of $P$. 
When $\dim \sigma = d$, i.e. $\sigma = \sigma_j$ for some $j=1,\ldots,m$, we define $\tilde{w}_{\sigma_j}$ so that 
$\supp \tilde{w}_{\sigma_j} \subset \interior \sigma_j$, and $\tilde{w}_\sigma \equiv 1$ on some neighborhood of $x_j$. 
Then there exists a continuous function $v: P \to [0,1]$ such that $x_1, \ldots, x_m \notin \supp v$ and 
$\tilde{w}_{\sigma_1}^{-1}(1) \cup \cdots \cup \tilde{w}_{\sigma_m}^{-1}(1) \cup v^{-1}(1)=P$. 

Next we define $\tilde{w}_\sigma$ when $\dim \sigma \le d-1$. 
By the induction hypothesis, one can take $w_\sigma: P^{(d-1)} \to [0,1]$ for each $\sigma \subset P^{(d-1)}$
so that our requirements (i)--(iii) hold for $(w_\sigma)_{\sigma \subset P^{(d-1)}}$.
We define $\tilde{w}_\sigma:P \to [0,1]$ by 
\[
\tilde{w}_\sigma(x):= \begin{cases} 
                       0 &( x \in \{x_1,\ldots, x_m \} ), \\
                      v(x) w_\sigma(r(x)) &(x \notin \{x_1, \ldots, x_m \}).
                      \end{cases}
\]
Let us check that $(\tilde{w}_\sigma)_\sigma$ satisfies our requirements (i)---(iii).
By definition, if $\dim \sigma = d$ then $\supp \tilde{w}_\sigma \subset \interior \sigma$. 
Then (i), (ii) are obvious when $\dim \sigma = d$. 
When $\dim \sigma \le d-1$, $\supp \tilde{w}_\sigma \subset r^{-1}(\supp w_\sigma)$. This is because
$\{\tilde{w}_\sigma \ne 0\}$ is contained in $\supp v \cap r^{-1}(\supp w_\sigma)$, which is closed in $P$. 
Then, one can prove (i) for $\dim \sigma \le d-1$ by 
\[
\supp \tilde{w}_\sigma \subset r^{-1}(\supp w_\sigma) \subset r^{-1}(\Star(\sigma) \cap P^{(d-1)}) \subset \Star(\sigma). 
\]
The second inclusion holds since $(w_\sigma)_\sigma$ satisfies (i), 
and the third inclusion holds since $r(\sigma_j \setminus \{x_j\}) = \partial \sigma_j$ 
for any $j=1,\ldots,m$. (ii) for $\dim \sigma \le d-1$ is proved as follows
(notice that $\supp w_\sigma \cap \supp w_{\sigma'} = \emptyset$, since $(w_\sigma)_\sigma$ satisfies (ii)): 
\[
\supp \tilde{w}_\sigma \cap \supp \tilde{w}_{\sigma'} \subset r^{-1}(\supp w_\sigma \cap \supp w_{\sigma'}) = \emptyset. 
\]
(iii) follows from $\bigcup_{\sigma \subset P^{(d-1)}} w_\sigma^{-1}(1) = P^{(d-1)}$ (i.e. $(w_\sigma)_\sigma$ satisfies (iii)) and 
$\tilde{w}_{\sigma_1}^{-1}(1) \cup \cdots \cup \tilde{w}_{\sigma_m}^{-1}(1) \cup v^{-1}(1)=P$. 
\end{proof}

\begin{lem}\label{lem:shortpath}
For any $R > r(Q)^2/2$ and 
$q \in Q$, there exists an open neighborhood $V$ of $q$ and a continuous map 
$\lambda: V \to \Lambda^R(Q)$ such that, 
there holds $\lambda(v)(0)=v$ and $\lambda(v)(1) \in \partial Q$ for any $v \in V$. 
\end{lem} 
\begin{proof}
Since $R > r(Q)^2 /2\ge \dist(q,\partial Q)^2/2$, 
there exists $\gamma \in \Lambda^R(Q)$ such that $\gamma(0)=q$ and $\gamma(1) \in \partial Q$. 
Then there exists an open neighborhood $\tilde{V}$ of $q$ and a continuous map $\tilde{\lambda}: \tilde{V} \to \Lambda(Q)$ such that 
$\tilde{\lambda}(q)=\gamma$ and $\tilde{\lambda}(v)(0)=v$, $\tilde{\lambda}(v)(1) \in \partial Q \, (\forall v \in \tilde{V})$. 
Then, $V:= \tilde{\lambda}^{-1}(\Lambda^R(Q))$ and $\lambda:= \tilde{\lambda}|_V$ satisfy our requirements. 
\end{proof} 

Before starting the proof of Proposition \ref{prop:bound}, we introduce some operations on $\Lambda(Q)$. 
\begin{itemize}
\item For any $a \in [0,1]$ and $\gamma \in \Lambda(Q)$, we define $a\gamma \in \Lambda(Q)$ by $a\gamma(t):=\gamma(at)$. 
The map $[0,1] \times \Lambda(Q) \to \Lambda(Q); (a,\gamma) \mapsto a\gamma$ is continuous. 
\item For any $\gamma \in \Lambda(Q)$, we define $\bar{\gamma} \in \Lambda(Q)$ by $\bar{\gamma}(t):=\gamma(1-t)$. 
The  map $\Lambda(Q) \to \Lambda(Q); \gamma \mapsto \bar{\gamma}$ is continuous. 
\item For any $\gamma_1, \ldots, \gamma_m \in \Lambda(Q)$ such that $\gamma_k(1)=\gamma_{k+1}(0)$ for $k=1,\ldots,m-1$, 
We define $\con(\gamma_1, \ldots,\gamma_m) \in \Lambda(Q)$ by 
\[
\con(\gamma_1,\ldots,\gamma_m)(t):= \gamma_{k+1}(m(t-k/m)) \quad (k/m \le t \le (k+1)/m,\, k=0,\ldots,m-1). 
\]
This is called the \textit{concatenation} of $\gamma_1,\ldots,\gamma_m$. 
The following map is continuous: 
\begin{align*}
&\{(\gamma_1,\ldots,\gamma_m) \mid \gamma_1,\ldots,\gamma_m \in \Lambda(Q), \gamma_k(1)=\gamma_{k+1}(0)\,(k=1,\ldots,m-1)\} \to \Lambda(Q); \\
& \qquad (\gamma_1,\ldots,\gamma_m) \mapsto \con(\gamma_1,\ldots,\gamma_m).
\end{align*}
\end{itemize}

\begin{proof}[\textbf{Proof of Proposition \ref{prop:bound}}]

First we prove $c^\Lambda(Q:\alpha) \le 2jr(Q)$. 
It is enough to show $I^{\Lambda,a}_0(\alpha)=0$ for any $a > (2jr(Q))^2/2$. 
Let us take a $j$-dimensional finite simplicial complex $P$, a subcomplex $P' \subset P$ and a continuous map $f:(P,P') \to (Q,\partial Q)$ such that
$\alpha \in f_*(H_j(P,P'))$. 

Suppose that there exists a continuous map $F: P \times [0,1] \to \Lambda^a(Q)$ which satisfies the following properties: 
\begin{enumerate}
\item[F-(i):]  For any $x \in P$, $F(x,0)=p_{f(x)}$. 
\item[F-(ii):]  For any $(x,t) \in P'':=P' \times [0,1] \cup P \times \{1\}$, $F(x,t) \in \Lambda_\partial (Q)$. 
\end{enumerate} 
We obtain the following commutative diagram, where $i^P: (P,P') \to (P \times [0,1], P'')$ is defined by $i^P(x):=(x,0)$. 
\[
\xymatrix{
H_j(P,P') \ar[r]^{f_*} \ar[d]_{(i^P)_*}  & H_j(Q,\partial Q) \ar[d]\ar[rd]^{I^{\Lambda,a}_0}& \\
H_j(P \times [0,1], P'') \ar[r]_-{F_*}  & H_j(\Lambda^a(Q),\Lambda^a(Q) \cap  \Lambda_\partial (Q))\ar[r]& H_j(\Lambda^a(Q) \cup \Lambda_\partial (Q), \Lambda_\partial(Q)). 
}
\]
It is easy to see that $(i^P)_*=0$, thus $I^{\Lambda,a}_0 \circ f_*=0$. 
Since $\alpha \in f_*(H_j(P,P'))$, we have $I^{\Lambda,a}_0(\alpha)=0$. 
Hence it is enough to define $F$ which satisfies F-(i) and F-(ii). 

By our assumption, $a/(2j)^2 > r(Q)^2/2$. 
By Lemma \ref{lem:shortpath}, for any $q \in Q$ there exists a neighborhood $V_q$ of $q$ and $\lambda_q: V_q \to \Lambda^{a/(2j)^2}(Q)$ which satisfies 
$\lambda_q(v)(0)=v$ and $\lambda_q(v)(1) \in \partial Q$ for any $v \in V_q$. 

By replacing $P$ with its subdivison if necessary, we may assume that the following holds: 
for any simplex $\sigma$ of $P$, there exists $q \in Q$ such that $f(\Star(\sigma)) \subset V_q$. 
We choose such $q$, and denote it by $q(\sigma)$. 
Moreover, we take $(w_\sigma)_\sigma$, a family of continuous functions on $P$ as in Lemma \ref{lem:polyhedron}. 

We define $F_k: P \to \Lambda(Q)$ for each $k=0,\ldots,j$. 
Since $(w_\sigma)_\sigma$ satisfies Lemma \ref{lem:polyhedron} (ii), for each $x \in P$ and $k=0,\ldots,j$,
either (a) or (b) holds: 
\begin{itemize}
\item[(a):] There exists a unique $k$-dimensional simplex $\sigma$ of $P$ such that $x \in \supp w_\sigma$. 
\item[(b):] $x \notin \supp w_\sigma$ for any $k$-dimensional simplex $\sigma$ of $P$. 
\end{itemize}
In case (a), $f(x) \in f(\Star(\sigma)) \subset V_{q(\sigma)}$. Then we define
$F_k(x) \in \Lambda(Q)$ by 
$F_k(x):=w_\sigma(x) \cdot \lambda_{q(\sigma)}(f(x))$, i.e.
\[
F_k(x): [0,1] \to Q; \quad t \mapsto \lambda_{q(\sigma)}(f(x)) (w_\sigma(x) \cdot t).
\]
In case (b), we define $F_k(x):= p_{f(x)}$. 
Then, it is easy to check that $F_k$ is a continuous map, which satisfies the following properties:
\begin{itemize}
\item For any $x \in P$, $\mca{E}(F_k(x)) < a/(2j)^2$. 
\item For any $x \in P$, $F_k(x)(0)=f(x)$. 
\item If $x \in P$ satisfies $w_\sigma(x)=1$ for some $k$-dimensional simplex $\sigma$ of $P$, $F_k(x)(1)= \lambda_{q(\sigma)}(f(x))(1) \in \partial Q$. 
\end{itemize} 

Now we define $F: P \times [0,1] \to \Lambda(Q)$ by 
\[
F(x,t):= \con(\overline{tF_0(x)}, tF_1(x), \overline{tF_1(x)}, \ldots, t F_{j-1}(x), \overline{tF_{j-1}(x)}, tF_j(x)).
\]
The above concatenation is well-defined, since $F_0(x)(0) = \cdots = F_j(x)(0)$. 
For any $x \in Q$, $\mca{E}(F_0(x)), \ldots, \mca{E}(F_j(x)) < a/(2j)^2$. 
Thus $\mca{E}(F(x,t))<a$, therefore $F(P \times [0,1]) \subset \Lambda^a(Q)$. 
For any $x \in P$ and $k=0,\ldots,j$, one has $0 \cdot F_k(x) = p_{f(x)}$, 
thus $F(x,0) = p_{f(x)}$. 
This shows that $F$ satisfies F-(i).

We check that $F$ satisfies F-(ii).
One has $F(x,t) \in \Lambda_\partial(Q)$ for any $(x,t) \in P' \times [0,1]$, 
since $F(x,t)(1/2j)=F_0(x)(0)=f(x) \in \partial Q$. 
Hence it is enough to show that $F(x,1) \in \Lambda_\partial (Q)$ for any $x \in P$. 
By Lemma \ref{lem:polyhedron} (iii), 
there exists a simplex $\sigma$ of $P$ such that $w_\sigma(x)=1$. Let $k:=\dim \sigma$. Then 
$F(x,1)(k/j) = F_k(x)(1) \in \partial Q$. 
Hence $F(x,1) \in \Lambda_\partial (Q)$. 
This completes the proof of $c^\Lambda(Q:\alpha) \le 2jr(Q)$. 

The proof of $c^\Omega(Q:\alpha) \le  2(j+1) r(Q)$ is almost same. 
Let us take $P' \subset P$ and  $f: (P, P') \to (Q, \partial Q)$ so that $\alpha \in f_*(H_j(P, P'))$. 
It is enough to show that, if $a/(2j+2)^2 > r(Q)^2/2$, there exists a continuous map $F': P \times [0,1] \to \Omega^a(Q)$ such that: 
\begin{enumerate}
\item[F'-(i):]  For any $x \in P$, $F'(x,0)=l_{f(x)}$. 
\item[F'-(ii):]  For any $(x,t) \in P''=P' \times [0,1] \cup P \times \{1\}$, $F'(x,t) \in \Omega_\partial (Q)$. 
\end{enumerate}
For each $k=0,\ldots,j$, we define $F'_k: P \to \Lambda^{a/(2j+2)^2}(Q)$ as in the proof of $c^\Lambda(Q:\alpha) \le 2jr(Q)$. 
Then we define $F'$ by 
\[
F'(x,t):= \con(tF'_0(x), \overline{tF'_0(x)}, \ldots,  t F'_j(x), \overline{tF'_j(x)}).
\]
Since $F'(x,t)(0)=f(x)=F'(x,t)(1)$, 
one can consider $F'(x,t)$ as an element in $\Omega(Q)$. 
It is easy to check $\mca{E}(F'(x,t)) < a$ for any $(x,t)  \in P \times [0,1]$, hence 
$F'(P \times [0,1]) \subset \Omega^a(Q)$. 
It is also easy to check that $F'$ satisfies F'-(i), (ii), in a similar way as in the 
proof of $c^\Lambda(Q:\alpha) \le 2jr(Q)$. 
\end{proof}

\subsection{Proof of Corollary \ref{cor:gdloop}}
We conclude this section with a proof of Corollary \ref{cor:gdloop}. 

\begin{proof}
The case $j=1$ is easy and omitted (see \cite{Rot}, pp.501--502).
Hence we may assume that $M$ is simply connected. 
By the Hurewicz theorem, it is enough to show that, if $H_j(M) \ne 0$ then there exists a nontrivial geodesic loop at $p$, 
of length $\le 2j \diam(M)$. 

Let $\rho(M)$ be the injectivity radius of $M$. 
For any $\ep < \rho(M)$, let 
$Q_\ep:= \{ x \in M \mid \dist(x,p) \ge \ep \}$. 
Then, it is clear that $r(Q_\ep) \le \diam(M) -\ep < \diam (M)$. 
Moreover, $H_j(Q_\ep, \partial Q_\ep) \cong H_j(M) \ne 0$. 

We apply Theorem \ref{thm:short} for $Q_\ep$. Then, there exists a brake billiard trajectory 
$\gamma_\ep$ on $Q_\ep$, such that $\len(\gamma_\ep) \le 2j r(Q_\ep) < 2j \diam(M)$. 
We set $\tau_\ep:= \min \{ t >0 \mid \gamma_\ep(t) \in \partial Q_\ep \}$, and define 
$\Gamma_\ep:[0,1] \to Q_\ep$ by  $\Gamma_\ep(t):= \gamma_\ep(\tau_\ep t)$. 
Since $\Gamma_\ep(0), \Gamma_\ep(1) \in \partial Q_\ep$ and 
$\len (\Gamma_\ep) \le \len(\gamma_\ep) < 2j \diam(M)$, a certain subsequence of 
$(\Gamma_\ep)_\ep$ converges to a geodesic loop $\Gamma:[0,1] \to M$ at $p$ such that 
$\len(\Gamma) \le 2j \diam (M)$. 

We have to check that $\Gamma$ is nonconstant. 
Since $\dot{\Gamma}_\ep(0)$ is perpendicular to $\partial Q_\ep$ and nonzero, 
$\Gamma_\ep([0,1])$ intersects $S:=\{ x \in M \mid \dist(x,p) = \rho(M) \}$. 
Hence $\Gamma([0,1])$ also intersects $S$. Since $p \notin S$, $\Gamma$ is nonconstant. 
\end{proof}

\section{Shortest periodic billiard trajectory on a convex body}

The goal of this section is to prove Theorem \ref{thm:BruMin} and Theorem \ref{thm:Ghomi} by our method. 
A recent paper \cite{ABKS} obtains similar proofs based on \cite{BB}. 
Although several results in this section were already obtained in \cite{BB}, here we include them for the sake of 
completeness.

First let us introduce some notations. 
Let $K \subset \R^n$ be a convex body with smooth boundary.   
\begin{itemize}
\item
We abbreviate $c^\Omega(K: [K, \partial K])$ as $c^\Omega(K)$. 
\item
$\mca{P}(K)$ denotes the set of periodic billiard trajectories on $K$. 
\item
$\mca{P}^+(K)$ denotes the set consisting of piecewise geodesic curves $\gamma:S^1 \to \R^n$ such that
$\gamma(S^1)+x \not \subset \interior K$ for any $x \in \R^n$. 
\item
For any $\nu \in \R^n$ and compact set $S \subset \R^n$, 
$h(S:\nu):= \max\{ s \cdot \nu \mid s \in S\}$. 
\item
For any $q \in \partial K$, $\nu(q)$ denotes the unit vector which is outer normal to $\partial K$ at $q$.
\end{itemize} 

\begin{lem}\label{lem:convex-1} 
Let $K$ be a convex body with smooth boundary, and 
$\gamma:S^1 \to \R^n$ be a piecewise geodesic curve. 
If there exists $\mca{N} \subset \R^n \setminus \{(0,\ldots,0)\}$ such that 
$(0,\ldots,0) \in \conv(\mca{N})$ and $h(K:\nu) \le h(\gamma(S^1):\nu)$ for any $\nu \in \mca{N}$, 
then $\gamma \in \mca{P}^+(K)$. 
\end{lem}
\begin{proof} 
Take $x \in \R^n$ arbitrarily. 
Since $(0,\ldots,0) \in \conv(\mca{N})$, 
there exists $\nu \in \mca{N}$ such that $x \cdot \nu \ge 0$. 
Thus $h(\gamma(S^1)+x : \nu) \ge h(\gamma(S^1):\nu) \ge h(K:\nu)$. 
Since $\nu \ne 0$, this shows that $\gamma(S^1) + x \not \subset \interior K$. 
\end{proof} 

\begin{lem}\label{lem:convex-2}
Any $\gamma \in \mca{P}(K)$ satisfies the assumption in Lemma \ref{lem:convex-1} with 
$\mca{N}:= \{ \nu(\gamma(t)) \mid t \in \mca{B}_\gamma\}$.
In particular, $\mca{P}(K) \subset \mca{P}^+(K)$. 
\end{lem} 
\begin{proof} 
For any $t \in \mca{B}_\gamma$, 
there holds $h(K:\nu(\gamma(t)))= \gamma(t) \cdot \nu(\gamma(t))$ since $K$ is convex. 
Hence $h(K:\nu) = h(\gamma(S^1):\nu)$ for any $\nu \in \mca{N}$. 

Suppose that $(0,\ldots,0) \notin \conv(\mca{N})$. 
Since $\mca{N}$ is a finite set, 
there exists $x \in \R^n$ such that $x \cdot \nu >0$ for any $\nu \in \mca{N}$. 
Since $\ddot{\gamma} \equiv 0$ on $S^1 \setminus \mca{B}_\gamma$, there exists $t \in \mca{B}_\gamma$ such that 
$x \cdot (\dot{\gamma}^-(t) - \dot{\gamma}^+(t)) \le 0$. 
On the other hand, it is easy to see that $\nu(\gamma(t))= \dot{\gamma}^-(t) - \dot{\gamma}^+(t)/|\dot{\gamma}^-(t) - \dot{\gamma}^+(t)|$. 
Thus we have $x \cdot \nu(\gamma(t)) \le 0$. 
This is a contradiction, thus $(0,\ldots,0) \in \conv(\mca{N})$. 
\end{proof} 

The following proposition is a key step in the proof. 

\begin{prop}\label{prop:convex-3}
For any $\gamma \in \mca{P}^+(K)$, there holds $c^\Omega(K) \le \len(\gamma)$. 
\end{prop}
\begin{proof}
It is enough to show that, for any $a > \len(\gamma)^2/2$ and $\delta>0$, 
the homomorphism
\[
H_n( \interior K, \interior K \setminus K(\delta)) \to 
H_n( \Omega^a(\interior K), \Omega^a(\interior K) \cap \Omega_\delta(\interior K)) 
\]
is zero. By the excision property, this is equivalent to show that
\[
H_n(\R^n, \R^n \setminus K(\delta)) \to 
H_n(\Omega^a(\R^n), \Omega^a(\R^n) \setminus \Omega(K(\delta)))
\]
is zero. 
By reparametrization of $\gamma$, we may assume that 
$\mca{E}(\gamma) = \len(\gamma)^2/2$, thus $\mca{E}(\gamma)<a$. 

Let us set $B_R:= \{ x \in \R^n \mid |x| \le R \}$ for any $R>0$. 
We define 
$F: B_R \times [0,1] \to \Omega^a(\R^n)$ by 
$F(w,s)(t):= w + s \gamma(t)$. 
When $R$ is sufficiently large, $w+ s \gamma(S^1) \not \subset K(\delta)$ for any $w \in \partial B_R$ and $0 \le s \le 1$. 
Moreover, $w + \gamma(S^1)  \not\subset K(\delta)$ for any $w \in B_R$, 
since $\gamma \in \mca{P}^+(K)$. 
Thus, setting $P:= B_R \times [0,1]$, $P':= \partial B_R \times [0,1] \cup B_R \times \{1\}$, we have 
\[
F: (P,P') \to (\Omega^a(\R^n) , \Omega^a(\R^n) \setminus \Omega(K(\delta))). 
\] 
Setting 
$i: (B_R, \partial B_R) \to (P,P'); x \mapsto (x,0)$, 
we have the following commutative diagram: 
\[
\xymatrix{
H_n(B_R, \partial B_R) \ar[r]^{i_*}\ar[d]& H_n(P,P') \ar[d]^{F_*} \\
H_n(\R^n,\R^n \setminus K(\delta) ) \ar[r]  & H_n(\Omega^a(\R^n) , \Omega^a(\R^n) \setminus \Omega(K(\delta))). 
}
\]
Since $K(\delta)$ is also convex, the left vertical arrow is an isomorphism. 
On the other hand, $i_*=0$. Thus the bottom homomorphism is zero. 
\end{proof} 

\begin{cor}\label{cor:convex-4}
Let us define $\mu_P^+(K):= \inf \{ \len(\gamma) \mid \gamma \in \mca{P}^+(K) \}$. 
Then $c^\Omega(K)=\mu_P(K)=\mu_P^+(K)$. 
\end{cor}
\begin{proof}
Lemma \ref{lem:spectrality} shows $c^\Omega(K) \ge \mu_P(K)$. 
$\mca{P}(K) \subset \mca{P}^+(K)$ shows $\mu_P(K) \ge \mu_P^+(K)$. 
Proposition \ref{prop:convex-3} shows $\mu_P^+(K) \ge c^\Omega(K)$. 
\end{proof} 
\begin{rem}\label{rem:shortest}
The identity $c^\Omega(K)=\mu_P(K)$ implies that there exists a shortest periodic billiard trajectory on $K$, 
since Lemma \ref{lem:spectrality} shows that there exists a periodic billiard trajectory $\gamma$ on $K$ such that $\len(\gamma)=c^\Omega(K)$. 
\end{rem}

The identity $\mu_P(K)=\mu_P^+(K)$ can be considered as a variational characterization of $\mu_P$. 
The same result is established in \cite{BB} (see also \cite{ABKS}, Theorem 2.1). 
As an immediate consequence, we can recover the following result, which was already obtained in Proposition 1.4 \cite{ArOs}
(see also \cite{ABKS} Section 2.2). 

\begin{cor}[\cite{ArOs}]\label{cor:convex-4.5}
Let $K_1 \subset K_2$ be convex bodies with smooth boundaries. Then $\mu_P(K_1) \le \mu_P(K_2)$. 
\end{cor}
\begin{proof}
It is obvious that $\mca{P}^+(K_2) \subset \mca{P}^+(K_1)$. Then we have 
$\mu_P(K_1) = \mu_P^+(K_1) \le \mu_P^+(K_2) = \mu_P(K_2)$.
\end{proof} 

We also need the following Lemma \ref{lem:convex-5} 
to determine when equality holds in Theorem \ref{thm:BruMin} and Theorem \ref{thm:Ghomi}.

\begin{lem}\label{lem:convex-5}
Suppose that $\gamma \in \mca{P}^+(K)$ satisfies $\len(\gamma)=\mu_P^+(K)$, and 
$|\dot{\gamma}(t)|$ is constant for all $t$ such that $\dot{\gamma}(t)$ exists. 
Then, up to parallel displacement, one has $\gamma \in \mca{P}(K)$. 
\end{lem}
\begin{proof}
For any $\ep>0$, we set $\gamma_\ep(t):= (1-\ep)\gamma(t)$. 
Since $\len(\gamma_\ep) < \len(\gamma) = \mu_P^+(K)$, one has $\gamma_\ep \notin \mca{P}^+(K)$. 
There exists $x_\ep \in \R^n$ such that $x_\ep + \gamma_\ep(S^1) \subset \interior K$ for any $\ep>0$, 
thus by parallel displacement we may assume $\gamma(S^1) \subset K$. 
We show that $\gamma \in \mca{P}(K)$. 

Take $0=t_0 < t_1< \cdots < t_m=1$ so that 
$\gamma|_{[t_{j-1},t_j]}$ are geodesics and 
$\dot{\gamma}^-(t_j) \ne \dot{\gamma}^+(t_j)$ for all $j$. 
We set $J:= \{ 1 \le j \le m \mid \gamma(t_j) \in \partial K\}$. 
For each $j \in J$, let us abbreviate $\nu(\gamma(t_j))$ as $\nu_j$. 
By convexity of $K$, $h(K:\nu_j) = \gamma(t_j) \cdot \nu_j$ for each $j \in J$. 

Let $\mca{N}:= \{\nu_j \mid j \in J\}$. 
If $(0,\ldots,0) \notin \conv(\mca{N})$, there exists $x \in \R^n$ such that $x \cdot \nu_j <0$ for any $j \in J$. 
Thus $\gamma(S^1) + cx \subset \interior K$ for sufficiently small $c>0$. 
This is impossible since $\gamma \in \mca{P}^+(K)$. 
Thus we have shown $(0,\ldots,0) \in \conv(\mca{N})$. 

We show that $J=\{1,\ldots,m\}$. 
If $J \subsetneq \{1,\ldots,m\}$, there exists $\gamma': S^1 \to K$ such that 
$\len(\gamma') < \len(\gamma)$ and $\gamma'(S^1) \supset \{ \gamma(t_j) \mid j \in J\}$. 
For each $j \in J$, one has 
\[
h(K: \nu_j ) = \gamma(t_j) \cdot \nu_j  \le h(\gamma'(S^1): \nu_j ).
\] 
Then Lemma \ref{lem:convex-1} implies $\gamma' \in \mca{P}^+(K)$. 
This is impossible since $\gamma$ has the shortest length in $\mca{P}^+(K)$. 

To prove $\gamma \in \mca{P}(K)$, it is enough to check that $\gamma$ satisfies the law of reflection at every $t_j$. 
If this is not the case, i.e. $\dot{\gamma}^+(t_j) - \dot{\gamma}^-(t_j)$ is not a multiple of 
$\nu_j$ for some $j$, there exists 
$v \in T_{\gamma(t_j)} \partial K$ such that 
\[
|\gamma(t_j) - \gamma(t_{j-1})| + |\gamma(t_{j+1}) - \gamma(t_j)| > 
|\gamma(t_j)+v - \gamma(t_{j-1})| + |\gamma(t_{j+1}) - \gamma(t_j)-v|.
\]
Define $\gamma':S^1 \to \R^n$ so that 
$
\gamma'(t_i) = \begin{cases} 
                \gamma(t_j) + v &(i=j) \\
                \gamma(t_i) &(i \ne j)
               \end{cases}
$
and $\gamma'|_{[t_{i-1},t_i]}$ are geodesics for all $1 \le i \le m$. 
Then $\len(\gamma') < \len(\gamma)$. 
It is easy to check $h( \gamma'(S^1): \nu_i) \ge h(K: \nu_i)$ for any $1 \le i \le m$, thus Lemma \ref{lem:convex-1} implies $\gamma' \in \mca{P}^+(K)$.
This is impossible since  
$\gamma$ has the shortest length in $\mca{P}^+(K)$.  
\end{proof} 

For any two curves $\gamma_i: S^1 \to \R^n \,(i=1,2)$, we define $\gamma_1+\gamma_2:S^1 \to \R^n$ 
by $\gamma_1+\gamma_2(t):= \gamma_1(t)+\gamma_2(t)$. 
The following lemma would be clear from the definition of $\mca{P}^+$. 

\begin{lem}\label{lem:convex-6}
If $\gamma_i(S^1) \notin \mca{P}^+(K_i)$ for $i=1,2$, 
one has $\gamma_1 + \gamma_2 \notin \mca{P}^+(K_1+K_2)$. 
\end{lem}

Now we can prove Theorem \ref{thm:BruMin}. 

\begin{proof}[\textbf{Proof of Theorem \ref{thm:BruMin}}]
Let $a_j:= \frac{\mu_P(K_j)}{\mu_P(K_1)+\mu_P(K_2)}$. 
If $\len(\gamma) < \mu_P(K_1) + \mu_P(K_2)$, we have the following inequality for each $j=1,2$:
\[
\len (a_j \gamma) = a_j \cdot \len(\gamma) < \mu_P(K_j) = \mu_P^+(K_j). 
\]
Then $a_j \gamma \notin \mca{P}^+(K_j)$.
By Lemma \ref{lem:convex-6}, 
$\gamma = a_1 \gamma + a_2 \gamma \notin \mca{P}^+(K_1+K_2)$. 
Thus we have shown that $\mu_P^+(K_1+K_2) \ge \mu_P(K_1)+ \mu_P(K_2)$. 
By Corollary \ref{cor:convex-4} we get
\begin{equation}\label{eq:BruMin}
\mu_P(K_1+K_2) = \mu_P^+(K_1+K_2) \ge \mu_P(K_1)+ \mu_P(K_2).
\end{equation}

We have to show that the following two conditions are equivalent: 

\begin{enumerate}
\item[(i):] $\mu_P(K_1+K_2) =\mu_P(K_1)+\mu_P(K_2)$. 
\item[(ii):] There exists a closed curve $\gamma$ which, up to parallel displacement and scaling, 
is a shortest periodic billiard trajectory on both $K_1$ and $K_2$. 
\end{enumerate} 

(i)$\implies$(ii): There exists $\gamma \in \mca{P}(K_1+K_2)$ such that $\len(\gamma) = \mu_P(K_1+K_2)$. 
If $a_1 \gamma \notin \mca{P}^+(K_1)$, one has $(a_1+\ep) \gamma \notin \mca{P}^+(K_1)$ for sufficiently small $\ep>0$. 
On the other hand, $(a_2-\ep) \gamma \notin \mca{P}^+(K_2)$ since 
$(a_2-\ep) \len(\gamma) < \mu_P(K_2)$. 
Thus $\gamma \notin \mca{P}^+(K_1+K_2)$, this is a contradiction. 
Therefore $a_1 \gamma \in \mca{P}^+(K_1)$. 
Since $\len(a_1 \gamma)= \mu_P(K_1)$, 
Lemma \ref{lem:convex-5} implies 
$a_1 \gamma \in \mca{P}(K_1)$ up to parallel displacement. 
We can prove $a_2 \gamma \in \mca{P}(K_2)$ in the same way, thus (ii) holds. 

(ii)$\implies$(i): 
Take $\gamma:S^1 \to \R^n$ as in (ii). 
For $j=1, 2$, let $\gamma_j$ be a shortest periodic billiard trajectory on $K_j$, which is obtained by 
parallel displacement and scaling of $\gamma$. 
We may assume that $\gamma = \gamma_1 + \gamma_2$. 
Then $\len(\gamma)= \len(\gamma_1) + \len(\gamma_2)= \mu_P(K_1) + \mu_P(K_2)$. 

It is easy to see that $\mca{B}_{\gamma_1} = \mca{B}_{\gamma_2}$. Let us denote it as $\mca{B}$. 
For each $t \in \mca{B}$, $\nu(t):= \dot{\gamma}^-(t)-\dot{\gamma}^+(t)/|\dot{\gamma}^-(t)-\dot{\gamma}^+(t)|$ 
is outer normal to $\partial K_j$ at $\gamma_j(t)$ for $j=1, 2$. 
Let us set $\mca{N}:= \{ \nu(t) \mid t \in \mca{B} \}$. 
By Lemma \ref{lem:convex-2}, we have 
$(0,\ldots,0) \in \conv(\mca{N})$ and $h(K_j:\nu)= h(\gamma_j(S^1):\nu)$ for any $\nu \in \mca{N}$, $j=1,2$.
Then, for any $\nu \in \mca{N}$ 
\[
h(K_1+K_2: \nu ) = h(K_1: \nu) + h(K_2:\nu) = h(\gamma_1(S^1):\nu) + h(\gamma_2(S^1):\nu) = h(\gamma(S^1):\nu).
\]
By Lemma \ref{lem:convex-1}, $\gamma \in \mca{P}^+(K_1+K_2)$. 
Hence 
\[
\mu_P(K_1+K_2) = \mu_P^+(K_1+K_2)  \le \len(\gamma) = \mu_P(K_1)+ \mu_P(K_2).
\]
Combined with (\ref{eq:BruMin}), (i) is proved. 
\end{proof} 

To prove Theorem \ref{thm:Ghomi}, we need the following lemma.

\begin{lem}\label{lem:Ghomi}
Let $B$ be a ball in $\R^n$ with radius $r>0$. 
Then, any $\gamma \in \mca{P}(B)$ satisfies $\len(\gamma) \ge 4r$, and equality holds if and only if $\gamma$ is a bouncing ball orbit. 
In particular, $\mu_P(B)=4r$.
\end{lem}
\begin{proof}
Let $k:= \sharp \mca{B}_\gamma$. Then, one has $\len(\gamma) = 2kr \sin(\pi j/k)$ for some $1 \le j \le k-1$.
Then, the lemma follows from short computations. 
\end{proof} 

\begin{proof}[\textbf{Proof of Theorem \ref{thm:Ghomi}}]
Let $K$ be a convex body, and $B$ be a largest ball contained in $K$. 
Since the radius of $B$ is $r(K)$, Corollary \ref{cor:convex-4.5} and Lemma \ref{lem:Ghomi} imply 
$\mu_P(K) \ge \mu_P(B) = 4r(K)$. 

Suppose that $\mu_P(K)=4r(K)$, and let $\gamma$ be a shortest periodic billiard trajectory on $K$. 
Then $\gamma \in \mca{P}(K) \subset \mca{P}^+(K) \subset \mca{P}^+(B)$, and 
$\len(\gamma) = 4r(K) = \mu_P(B)$. 
Then Lemma \ref{lem:convex-5} shows that $\gamma \in \mca{P}(B)$ up to parallel displacement. 
By Lemma \ref{lem:Ghomi}, $\gamma$ is a bouncing ball orbit. 
In particualr, $\gamma$ is orthogonal to $\partial K$ at bouncing points. 
Thus $K$ is contained in a slab of thickness $\len(\gamma)/2=2r(K)$. 
Hence $\wid (K) = 2r(K)$. 

Suppose that $\wid(K)= 2r(K)$. 
Then $K$ is contained in a slab $S$ of thickness $2r(K)$. 
Let $\gamma$ be a bouncing ball orbit on $S$, i.e. $\gamma$ is a shortest orbit which touches both connected components of $\partial S$. 
Then, it is easy to see that $\gamma \in \mca{P}^+(S) \subset \mca{P}^+(K)$. 
Thus $\mu_P(K) =\mu_P^+(K) \le \len(\gamma)=4 r(K)$. 
\end{proof}

\textbf{Acknowledgements.} 
The author would like to thank an anonymous referee for many useful comments. 
The author is supported by JSPS KAKENHI Grant No. 25800041.

\end{document}